\documentclass[12pt,a4paper]{amsart}
\usepackage{pstricks}
\usepackage{amssymb,amsmath,stmaryrd,booktabs}
\usepackage{epstopdf}
\usepackage{graphicx}
\usepackage{setspace}

\newtheorem{theorem}{Theorem}[section]
\newtheorem{lemma}[theorem]{Lemma}
\newtheorem{corollary}[theorem]{Corollary}
\newtheorem{proposition}[theorem]{Proposition}

\newtheorem{remark}{Remark}
\newtheorem{question}[theorem]{Question}

\def\auto{{\phi}}

\def\sda{\shortdownarrow}
\def\ssda{\hbox{$\scriptscriptstyle\shortdownarrow$}}
\def\sua{\shortuparrow}
  \def\Ht{\hbox{\rm Ht}}
  \def\rev{\hbox{\rm rev}}

\def\Di{\mbox{$\Delta$}}

\def\hom{\mbox{$\psi$}}
\def\GG{{\mathbf{G}}}
\def\PP{{\mathcal{P}}}
\def\RR{{\mathcal{R}}}
\def\BB{{\mathbf{B}}}
\def\PPo{{\PP^\oplus}}
\def\BBo{{\BB^\oplus}}
\def\BBc{{\BB^+}}
\def\BBd{{\BB^\times}}
\def\SS{S_r}
\def\SSM{S^{-}_r}
\def\TT{T}
\def\TE{\TT_e}

\def\bw{{\mathbf{w}}}
\def\bc{{\mathbf{c}}}

\newgray{vlgray}{.9}
\newrgbcolor{darkgreen}{0 0.5 0} 
  \def\<{\langle}
  \def\>{\rangle}

  \def\t{\tau}

\def\R{\mbox{$\mathcal R$}}

\def\A{\mbox{$\mathcal A$}}

 \def\Z{\mathbb Z}
\def\N{\mathbb N}
\def\R{\mathbb R}
\def\C{\mathbb C}

\def\revsto{\curvearrowright}

\def\mod{\operatorname{mod}}

\def\cleq{\mbox{$\, \preccurlyeq\, $}}
\def\cgeq{\mbox{$\, \succcurlyeq\, $}}

\def\resp{\hbox{\emph{resp.} }}

  \def\revstorelR{\,\mbox{$\curvearrowright_{\scriptscriptstyle{\mathcal R}}$}}

\title[A new Garside structure for braid groups of~type~$(e,e,r)$]
{A new Garside structure\\ for braid groups of~type~$(e,e,r)$}
\author{Ruth Corran}
\address{R. Corran: The American University of Paris, 147 rue de Grenelle, 75007 Paris}
\email{Ruth.Corran@aup.fr}
\author{Matthieu Picantin}
\address{M. Picantin: LIAFA UMR 7089 CNRS et Universit\'e Paris 7 Denis Diderot,
2 place Jussieu, Case 7014, F-75251 Paris Cedex 05}
\email{picantin@liafa.jussieu.fr}

\begin{document}
\date{\today}
\begin{abstract} We describe a new presentation for the complex reflection groups of type~$(e,e,r)$ and their braid groups. A diagram for this presentation is proposed. 
The presentation is a monoid presentation which is shown to give rise to a Garside structure.
A detailed study of the combinatorics of this structure leads us to describe it as~\emph{post-classical}.
\end{abstract}
\maketitle

\section{ Introduction}
\label{s:intro}

A complex reflection group is a group acting on a finite-dimensional complex vector space, that is generated by complex reflections: non-trivial elements that fix a complex hyperplane in space pointwise. Any real reflection group becomes a complex reflection group if we extend the scalars from~$\R$ to~$\C$. In particular all Coxeter groups or Weyl groups give examples of complex reflection groups, although not all complex reflection groups arise in this way. One would like to generalise as much as possible from the theory of Weyl groups and Coxeter groups to complex reflection groups.

For instance, according to~Brou\'e--Malle--Rouquier~\cite{bmr}, one can define the braid group~$\BB(W)$ attached to a complex reflection group~$\GG(W)$ as the fundamental group of the space of regular orbits. When~$\GG(W)$ is real, the braid group~$\BB(W)$ is well understood owing to Brieskorn's
presentation theorem and the subsequent structural study by Deligne and Brieskorn--Saito~\cite{brieskorn,del,brsa}: their main combinatorial results express that $\BB(W)$ is the group of fractions of a monoid in which divisibility has good properties, and, in addition, there exists a distinguished element whose divisors encode the whole structure: in modern terminology, such a monoid is called Garside.
The group of fractions of a Garside monoid is called a Garside group. Garside groups enjoy many
remarkable group-theoretical, cohomological and homotopy-theoretical properties. Finding (possibly various) Garside structures for a given group becomes a natural challenge.

\subsection{ The groups~$\GG(e,e,r)$ and~$\BB(e,e,r)$}
\label{ss:groups}

The classification of (irreducible) finite complex reflection groups was obtained 
by~Shephard and~Todd \cite{st}:
\begin{itemize}
\item an infinite family~$\GG(de, e, r)$ where $d,e,r$ are arbitrary positive integral parameters;
\item 34 exceptions, labelled~$\GG_4,\ldots,\GG_{37}$.
\end{itemize}

The infinite family includes the four infinite families of finite Coxeter groups: $\GG(1, 1, r)\sim\GG(A_{r-1})$, $\GG(2, 1, r)\sim\GG(B_r)$, $\GG(2, 2, r)\sim\GG(D_r)$ and~$\GG(e, e, 2)\sim\GG(I_2(e))$. For
all other values of the parameters, $\GG(de, e, r)$ is an irreducible monomial complex
reflection group of rank~$r$, with no real structure.

In the infinite family, one may consider, in addition to the real groups, the complex subfamily~$\GG(e, e, r)$---note that
this subseries contains the $D$-type and~$I_2$-type Coxeter series---and our objects of interest are the possible Garside structures for the braid group~$\BB(e, e, r)$.

The reflection groups of type~$(e,e,r)$ are defined in terms of positive integral parameters~$e,r$:
$$\GG(e,e,r) = \left\{ 
\begin{array}{cc}
r \times r \mbox{ monomial matrices } \\
(x_{ij}) \mbox{ over } \{0\} \cup \mu_{e} 
\end{array}
\left | \ 
\prod_{x_{ij}\neq 0} x_{ij} = 1 \right. \right\},$$
\vbox{that is, as the group of~$r \times r$ matrices consisting of:\begin{list}{$\bullet$}{\leftmargin=20pt}

\item monomial matrices (each row and column has a unique non-zero entry), 
\item with all non-zero entries lying in~$\mathbf{\mu}_{ e}$, the~$e$-th roots of unity, and
\item for which the product of the non-zero entries is 1.
\end{list}}

The group~$\GG(e,e,r)$ is generated by reflections of~$\C^r$.
There are hyperplanes  in~$\C^r$ corresponding to the reflections of the reflection group.

The corresponding \emph{braid group}~$\BB(e,e,r)$ is defined in terms of the 
fundamental group of a quotient of the hyperplane complement. 
We do not make recourse to this definition; our starting point will be
known presentations for these braid groups. 

\subsection{ Brou\'e--Malle--Rouquier presentation}
\label{ss:bmr}

\def\sBraid{(R_1)}
\def\sComm{(R_2)}
\def\eLink{(P_1)}
\def\stBraid{(P_2)}
\def\stComm{(P_3)}
\def\dbleBar{(P_4)}
\def\stBraidZ{(R_3)}
\def\stCommZ{(R_4)}
\def\tCircle{(R_5)}

Such a presentation for the braid group~$\BB(e,e,r)$ may be found in~\cite{bmr}:
\label{BMRpres}
\begin{itemize}
\item
Generators: $\{t_0,t_1\} \cup \SS$ with~$\SS = \{s_3,\ldots,s_r\}$, 
and 
\item 
Relations: 
$$\begin{array}{rcl}
\sBraid	& s_i s_j s_i = s_j s_i s_j	&\mbox{ for }|i-j|=1,\\
\sComm	& s_i s_j = s_j s_i		&\mbox{ for }|i-j|>1,\\
\eLink	& \<t_1 t_0\>^e =  \<t_0 t_1\>^{e} \\
\stBraid	& s_3 t_i s_3 =t_i s_3 t_i & \mbox{ for  }i = 0,1, \\
\stComm	& s_j t_i = t_i s_j & \mbox{ for  } i = 0,1, \mbox{ and } 4 \leq j \leq r, \mbox{ and }\\
\dbleBar	& (s_3 t_1 t_0)^2 = (t_1 t_0 s_3)^2,\end{array}
$$
\end{itemize}
where~$\<ab\>^e$ denotes the alternating product of~$a$ and~$b$ with~$e$ terms. The collections of relations~$\sBraid$ and~$\sComm$ are the usual braid relations on those generators in~$\SS$. 

Furthermore, it is shown in~\cite{bmr} that by adding the relation~$a^2=1$ for all generators~$a$,  
a presentation for the reflection group~$\GG(e,e,r)$ is obtained. 
The generators in this case are all reflections in~$\GG(e,e,r)$.

\begin{figure}[h]
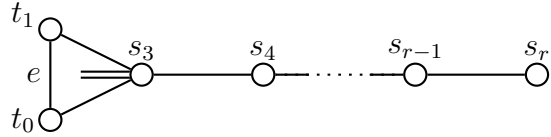

\vskip-8mm
\hskip-8.7cm
\psset{unit=4mm}
%
\qline(4,-6.5)(7,-5)
\qline(4,-3.5)(7,-5)
\qline(4,-6.5)(4,-3.5)
\uput{.3}[180](4,-5){$e$}
\qline(7,-5)(11,-5)
\qline(11,-5)(12.5,-5)
\psset{linestyle=dashed, dash=1pt 3pt, linecolor=black}
\qline(11,-5)(16,-5)
\psset{linestyle=solid}
\qline(14.5,-5)(16,-5)
\qline(16,-5)(20,-5)
\qline(7,-5.1)(5,-5.1)
\qline(7,-4.9)(5,-4.9)
\pscircle*[linecolor=white](4,-6.5){.4}
\pscircle(4,-6.5){.4}
\uput{.5}[180](4,-6.5){$t_0$}
\pscircle*[linecolor=white](4,-3.5){.4}
\pscircle(4,-3.5){.4}
\uput{.5}[180](4,-3,5){$t_1$}
\pscircle*[linecolor=white](7,-5){.4}
\pscircle(7,-5){.4}
\uput{.5}[90](7,-5){$s_3$}
\pscircle*[linecolor=white](11,-5){.4}
\pscircle(11,-5){.4}
\uput{.5}[90](11,-5){$s_4$}
\pscircle*[linecolor=white](16,-5){.4}
\pscircle(16,-5){.4}
\uput{.5}[90](16,-5){$s_{r-1}$}
\pscircle*[linecolor=white](20,-5){.4}
\pscircle(20,-5){.4}
\uput{.5}[90](20,-5){$s_{r}$}
\vskip30mm
\caption{The diagram of type~$(e,e,r)$ by~\cite{bmr}.}
\label{BeerBMRDiagram}
\end{figure}

A diagram shown in Figure~\ref{BeerBMRDiagram} is proposed in~\cite{bmr} for this presentation.
This diagram is interpreted,  where possible,  as a Coxeter diagram. The vertices correspond to generators, and the edges to relations:  for each pair of vertices~$a$ and~$b$,
\begin{list}{$\bullet$}{\leftmargin=18pt}
\item no edge connecting the vertices corresponds to a relation~$ab = ba$,
\item an unlabelled edge connecting the vertices corresponds to~$aba = bab$,
\item an edge labelled~$e$ connecting the vertices corresponds~$\<ab\>^e = \<ba\>^e$.
\end{list}
The first two of these give the usual braid relations and the relations~$(P_2)$ and~$(P_3)$; the third gives the relation~$(P_1)$.
It remains to interpret the triangle with short double-line in the interior; in the diagram above, this represents the relation~$(P_4)$: $s_3 (t_1 t_0) s_3 (t_1 t_0) = (t_1 t_0) s_3 (t_1 t_0) s_3$. (This would be a relation corresponding to an edge labelled~$4$ between nodes~$s_3$ and~$t_1 t_0$, if the latter were a node. Conventionally, edges labelled by 4 in Coxeter diagrams are designated by double-lines.)

In the case of finite real reflection groups---that is, finite Coxeter groups---an enormous amount of understanding about
the reflection and braid groups arises from the Coxeter presentations coming from 
the  choice of generators  corresponding to a simple set of roots in a 
root system. In this paper we describe presentations of the  reflection groups~$\GG(e,e,r)$
and their braid groups~$\BB(e,e,r)$ which have some properties like those of Coxeter presentations.

\subsection{ Classical \emph{vs} dual braid monoids}
\label{ss:versus}

The success of~\cite{bkl}, which describes an alternative braid monoid for the ordinary braid group~$\BB(A_{n-1})$, provided the impetus to unify the different approaches by introducing a general framework: the Garside theory (see~\cite{dfx,dehornoy:garside,bdm,bessisdual}). This terminology refers to the fact that, although dealing only with the ordinary $n$-strand braid group~$\BB(A_{n-1})$, the pioneer paper by~Garside~\cite{garside} stands out in which the foundation is laid for a more systematic study of the divisibility theory in a well-chosen submonoid of the braid group.

Garside structures (see~Subsection~\ref{ss:basicsGarside} for details) are desirable because they allow fast calculation in the group (solution to word and conjugacy problems) by convenient canonical or normal forms. A given Garside group admits possibly several Garside structures, each providing an associated biautomatic structure, etc. Known examples of Garside groups are braid groups, torus link groups, one-relator groups with center, etc. In the particular case which concerns us here, that is, in the case of braid groups, two Garside structures---when defined---seem to be most natural: we will use here the term of \emph{classical braid monoid} (short for Artin--Brieskorn--Deligne--Garside--Saito--Tits monoid) and the term of \emph{dual braid monoid} proposed by~Bessis in~\cite{bessisdual} (corresponding to those monoids studied in~\cite[...]{bkl,bessisdual,cras,bcArXiv,bc}. Given a reflection group~$\GG(de,e,r)$, when defined and when no confusion is possible, we will write~$\BBc(de,e,r)$ for the classical braid monoid and~$\BBd(de,e,r)$ for the dual braid monoid.

The presentation in~\cite{bmr} does \emph{not} give rise to a Garside structure\footnote{In particular, this presentation can be viewed as a monoid presentation; the associated monoid is not cancellative, so does not embed in a group (see~\cite{co,bc}).}. A presentation giving rise to a Garside monoid for~$\BB(e,e,r)$ was obtained in~\cite{bc}; this monoid fits into the context of dual braid monoids and will be denoted by~$\BBd(e,e,r)$. In that case, the generators are in bijection with the reflections in~$\GG(e,e,r)$. In this paper we introduce a new presentation for~$\BB(e,e,r)$ which again gives rise to a Garside monoid, but which has more in common with the classical braid monoids than the dual braid monoids.

\bigbreak

The organization of the rest of the paper is as follows. In Section~\ref{s:newPresentation}, a new presentation (with an associated diagram) for~$\BB(e,e,r)$ is shown (Theorem~\ref{beerpresthm}). In Section~\ref{s:newGarsideStructure} we prove that this new presentation gives rise to a Garside monoid~$\BBo(e,e,r)$ (Theorem~\ref{th:garside}). The underlying Garside structure is then investigated (Theorem~\ref{Simples}). Finally, this allows us to situate ~$\BBo(e,e,r)$ as well as possible with respect to the dichotomy between classical and dual braid monoids (Subsection~\ref{ss:howClassical}).

\section{ A new presentation}
\label{s:newPresentation}

In this section we first introduce the new presentation for the braid group~$\BB(e,e,r)$, propose a diagram for the presentation, and then discuss its relationship to the reflection group and to other braid groups. Finally, after considering the notion of circle,
we prove that the given presentation does present the group~$\BB(e,e,r)$.

\subsection{ New presentation of type~$(e,e,r)$}
\label{ss:eerPres}

Let~$\PPo(e,e,r)$ denote the presentation given by:
\begin{itemize}
\item[$\bullet$]
Generators: $\TE\cup\SS$ with~$\TE = \{t_i \mid i \in \Z/e \}$ and~$\SS = \{s_3, \ldots, s_r\}$, and
\item 
Relations: 
$$\begin{array}{rcl}
\sBraid	& s_i s_j s_i = s_j s_i s_j	&\mbox{for~$|i-j|=1$,} \\
\sComm	& s_i s_j = s_j s_i		&\mbox{for~$|i-j|>1$,} \\
\stBraidZ	& s_3 t_i s_3 = t_i s_3 t_i	&\mbox{for~$i \in \Z/e$,} \\
\stCommZ	& s_j t_i = t_i s_j		&\mbox{for~$i \in \Z/e$ and~$4\leq i\leq r$, and } \\
\tCircle	& t_{i} t_{i-1} = t_{j} t_{j-1}	&\mbox{for~$i, j \in \Z/e$.}
\end{array}
$$
\end{itemize}

We will show in Subsection~\ref{ss:isBeerPres}:

\begin{theorem}
\label{beerpresthm}
The presentation~$\PPo(e,e,r)$ is a group presentation for the braid group~$\BB(e,e,r)$.
Furthermore, adding the relations~$a^2=1$ for all generators~$a$ gives a presentation of 
the reflection group~$\GG(e,e,r)$. In particular, the generators of this presentation
are all reflections. 
\end{theorem}

The new generating set is a superset of the generating set of~\cite{bmr}. The new generators~$t_i$ for~$2 \leq i\leq e-1$ may be defined inductively by~$t_i = t_{i-1} t_{i-2} t_{i-1}^{-1},$ and so are just conjugates of the original generators.

The presentation~$\PPo(e,e,r)$ can be viewed as a monoid presentation. The corresponding monoid~$\BBo(e,e,r)$ will be the starting point for constructing the Garside structure for~$\BB(e,e,r)$, and we will see:

\begin{proposition}
\label{MonoidIsomorphism}
The submonoid of~$\BB(e,e,r)$ generated by~$\TE\cup\SS$ is isomorphic to the monoid~$\BBo(e,e,r)$, that is, it can be presented by~$\PPo(e,e,r)$ considered as a monoid presentation.
\end{proposition}à

\subsection{ New diagram of type~$(e,e,r)$}
\label{ss:diagram}

We propose the diagram shown in Figure~\ref{BeerDiagram} as a type~$(e,e,r)$ analogy to the Coxeter diagrams for the real reflection group case. 

\begin{figure}[h]
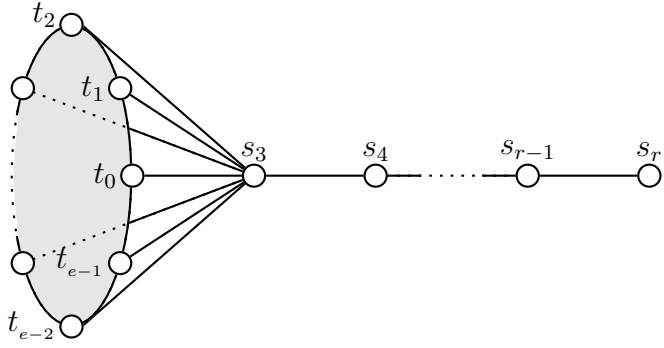

\vskip4mm
\hskip-8cm
\psset{unit=4mm}
\qline(1,-10.3)(7,-5)
\qline(1,0.3)(7,-5)
\qline(3,-5)(7,-5)
\qline(2.6,-7.9)(7,-5)
\qline(2.6,-2.1)(7,-5)
\qline(-0.6,-7.9)(7,-5)
\qline(-0.6,-2.1)(7,-5)
\qline(7,-5)(11,-5)
\qline(11,-5)(12.5,-5)
\psset{linestyle=dashed, dash=1pt 3pt, linecolor=black}
\qline(11,-5)(16,-5)
\psset{linestyle=solid}
\qline(14.5,-5)(16,-5)
\qline(16,-5)(20,-5)
\psellipse*[linecolor=vlgray](1,-5)(2,5)
\psellipse[linecolor=black](1,-5)(2,5)
\pscurve[linecolor=white, linewidth=1.6pt](-0.82,-3)(-0.97,-5)(-0.82,-7)
\psellipse[linestyle=dashed, dash=1pt 4pt,linecolor=black](1,-5)(2,5)
\psset{linestyle=dashed, dash=1pt 3pt, linecolor=black}
\qline(-0.6,-7.9)(7,-5)
\qline(-0.6,-2.1)(7,-5)
\psset{linestyle=solid}
\pscircle*[linecolor=white](1,0){.4}
\pscircle(1,0){.4}
\uput{.5}[160](1,0){$t_2$}
\pscircle*[linecolor=white](1,-10){.4}
\pscircle(1,-10){.4}
\uput{.5}[180](1,-10){$t_{_{e-2}}$}
\pscircle*[linecolor=white](3,-5){.4}
\pscircle(3,-5){.4}
\uput{.5}[180](3,-5){$t_0$}
\pscircle*[linecolor=white](2.6,-7.9){.4}
\pscircle(2.6,-7.9){.4}
\uput{.5}[180](2.6,-7.9){$t_{_{e-1}}$}
\pscircle*[linecolor=white](-0.6,-7.9){.4}
\pscircle(-0.6,-7.9){.4}
%
\pscircle*[linecolor=white](2.6,-2.1){.4}
\pscircle(2.6,-2.1){.4}
\uput{.5}[180](2.6,-2.1){$t_1$}
\pscircle*[linecolor=white](-0.6,-2.1){.4}
\pscircle(-0.6,-2.1){.4}
%
\pscircle*[linecolor=white](7,-5){.4}
\pscircle(7,-5){.4}
\uput{.5}[90](7,-5){$s_3$}
\pscircle*[linecolor=white](11,-5){.4}
\pscircle(11,-5){.4}
\uput{.5}[90](11,-5){$s_4$}
\pscircle*[linecolor=white](16,-5){.4}
\pscircle(16,-5){.4}
\uput{.5}[90](16,-5){$s_{r-1}$}
\pscircle*[linecolor=white](20,-5){.4}
\pscircle(20,-5){.4}
\uput{.5}[90](20,-5){$s_{r}$}
\vskip42mm
\caption{The new diagram of type~$(e,e,r)$: there are~$e$ nodes on the circle.}
\label{BeerDiagram}
\end{figure}

This diagram is again to be read as a Coxeter diagram  where possible, that is, when vertices~$a$ and~$b$ are joined by an (unlabelled) edge, there is a relation~$aba=bab$. 
The circle with~$e$ vertices at the left of the diagram corresponds to the \emph{circle}~$\{t_i \mid i \in \Z/e\}$ (see Subsection~\ref{sss:circles}).
Whenever two vertices~$a$ and~$b$ lie on this circle, there is a relation of the form
$aa^\sda=bb^\sda$ where~$a^\sda$ and~$b^\sda$ are the nodes immediately preceding~$a$ and~$b$ respectively on the circle.
If two nodes~$a$ and~$b$ are neither connected by an edge nor both lie on the disc, then there is 
a relation of the form~$ab=ba$---that is, the corresponding generators commute.

\textbf{The diagram automorphism~$\sda$ and its inverse~$\sua$.} 
Define the map~$\sda$ by~$s_j^\sda = s_j$ for all~$3 \leq j \leq r$ and~$t_i^\sda = t_{i-1}$ for all~$i \in \Z/e$. Since $\rho_1^\sda = \rho_2^\sda$ itself is a defining relation whenever $\rho_1 = \rho_2$ is a defining relation, then $\sda$ is a well-defined monoid morphism of~$\BBo(e,e,r)$. Furthermore, since
the whole set of relations defining~$\BBo(e,e,r)$ is stable under~$\sda$,
the map~$\sda$ is an automorphism of~$\BBo(e,e,r)$. The automorphism~$\sda$ rotates the circle
in the negative direction by a turn of~$\frac{2 \pi}{e}$.The same can be said for its inverse~$\sua$.
These diagram automorphisms give rise to automorphisms of the braid group~$\BB(e,e,r)$ as well as of the reflection group~$\GG(e,e,r)$. 
Moreover, these diagram automorphisms send (braid) reflections to (braid) reflections.

\begin{proposition}
The~$\sda$-trivial subgroup of the braid group~$\BB(e,e,r)$ is isomorphic to the braid group~$\BB(B_{r-1})$. 
\end{proposition}

\begin{proof} The proof follows the one of~\cite[Proposition~9.4]{dfx}.
\end{proof}

\textbf{The diagram anti-isomorphism~$\rev$.} Let~$\rev(\PPo(e,e,r))$ denote the presentation on the same generators as~$\PPo(e,e,r)$, and relations obtained by reversing all its relations. This presentation has a diagram corresponding to the mirror image of the diagram for~$\PPo(e,e,r)$. Let~$\rev(\BBo(e,e,r))$ be the monoid defined by~$\rev(\PPo(e,e,r))$.

\begin{lemma}
\label{MisomRevM}
The monoid~$\BBo(e,e,r)$ is isomorphic to~$\rev(\BBo(e,e,r))$
by the isomorphism~$\varphi$ which sends~$t_i \mapsto t_{-i}$ and~$s_j \mapsto s_j$.
\end{lemma}

\begin{proof}The map~$\varphi$ permutes the generators~$\TE\cup\SS$, and is bijective between the relations~of~$\PPo(e,e,r)$ and those of~$\rev(\PPo(e,e,r))$: the latter is clear for all types of relation { possibly }except~$(R_5)$, and in this case we find:
$$\varphi(t_i t_{i-1}) = t_{-i} t_{1-i} =_{\tiny \rev(\RR)} t_{-j} t_{1-j} = \varphi(t_j t_{j-1}).$$
So $\varphi$ is a well-defined monoid homomorphism, which is both surjective (as it permutes the generators) and injective (as it is bijective on the relations). Hence it is an isomorphism of monoids.
\end{proof}

Thus `mirror flipping' the diagram corresponds to a group isomorphism but not an equality.  Unlike for braid groups of Coxeter groups, this diagram morphism does not give rise to an automorphism of~$\BB(e,e,r)$.

\subsection{ Natural maps between different types}
\label{ss:maps}

Parabolic subgroups of (braid groups of) Coxeter groups may be realized by considering subdiagrams of the corresponding diagrams. We describe here parabolics of type~$(e,e,r)$ and the corresponding subdiagrams of the diagram shown on~Figure~\ref{BeerDiagram}, as well as maps which arise by taking diagram quotients instead. 

\subsubsection{ Maps related to parabolic subdiagrams}
\label{sss:parabolic}
{ Following~\cite{bmr}, for a given diagram, consider the equivalence relation on nodes defined by~$s\sim s$, and for~$s\not=t$
\[s\sim t\Leftrightarrow s \hbox{ and } t \hbox{ are not in a homogeneous relation with support }\{s,t\}.\]
Thus, for the diagram of Figure~\ref{BeerDiagram}, the equivalence classes have 1 or $e$ elements, and there is at most one class with~$e$ elements.

An \emph{admissible subdiagram} is a full subdiagram of the same type, that is, with 1 or~$e$ elements per class.

An admissible subdiagram of a diagram of type~$(e,e,r)$ must be of the form the union of a diagram of type~$(e_0,e_0,r_0)$ along with $k$ diagrams of type~$(1,1,r_i)$ where~$e_0\in\{0,1,e\}$ and~$\sum_{i=0}^kr_i\leq r$.}

Particular examples are considered below, which show the relationship with braid groups of some real reflection groups.

\begin{list}{$\bullet$}{\leftmargin=18pt}
\item $\PPo(e,e,r')$ with~$r' \leq r$: the case of `chopping off the tail' of the parachute. This corresponds to reducing the dimension from~$r$ to~$r'$. A special case of this is~$\PPo(e,e,2)$, where the whole tail is chopped off, leaving only the circle: this is a presentation of the dual braid monoid~$\BBd(I_2(e))$ (see Remark~\ref{R:dihedral} on page~\pageref{R:dihedral}).
\item $\PPo(0,0,r)$ is a presentation of the classical braid monoid~$\BBc(A_{r-2})$.
\item $\PPo(1,1,r)$ is a presentation of the classical braid monoid~$\BBc(A_{r-1})$.
\item $\PPo(2,2,r)$ is a presentation of the classical braid monoid~$\BBc(D_r)$. 
\end{list}

These sub-presentations will be used in Subsection~\ref{sss:completeness} in the context of cube condition calculations.

\subsubsection{ Maps related to foldings (diagram quotients)}
\label{sss:folding}

\begin{enumerate}
\item Epimorphism~$\BB(e_2, e_2, r) \twoheadrightarrow\BB(e_1, e_1, r)$ for~$e_1$ dividing~$e_2$.

\noindent 
The map induced by~$t_j \mapsto t_{j \mod {e_1}}$ and~$s_j \mapsto s_j$ defines an epimorphism
$\nu:\BB(e_{2},e_{2},r) \twoheadrightarrow\BB(e_1,e_1,r)$. There is an analogous map between the corresponding monoids and reflection groups.
This corresponds to a folding of the `parachute' part of the diagram.

\item Type~$B$ embedding: $\BB(2,1,r-1) \hookrightarrow\BB(e,e,r)$.

\noindent 
The type~$(2,1,r-1)$ corresponds to the Artin-Tits/Coxeter type~$B_{r-1}$.
The associated Coxeter diagram is:
\begin{center}
\begin{pspicture}(0,.8)(7,1.7)
%
\qline(0,.94)(1.6,.94)
\qline(0,1.06)(1.6,1.06)
\qline(1.6,1)(3.9,1)
\psset{linestyle=dashed, dash=1pt 3pt, linecolor=black}
\qline(3.9,1)(4.7,1)
\psset{linestyle=solid}
\qline(4.7,1)(7,1)
\pscircle*[linecolor=white](0,1){.18}
\pscircle(0,1){.18}
\uput{.3}[90](0,1){$q_1$}
\pscircle*[linecolor=white](1.6,1){.18}
\pscircle(1.6,1){.18}
\uput{.3}[90](1.6,1){$q_2$}
\pscircle*[linecolor=white](3.2,1){.18}
\pscircle(3.2,1){.18}
\uput{.3}[90](3.2,1){$q_3$}
\pscircle*[linecolor=white](5.4,1){.18}
\pscircle(5.4,1){.18}
\uput{.3}[90](5.4,1){$q_{r-2}$}
\pscircle*[linecolor=white](7,1){.18}
\pscircle(7,1){.18}
\uput{.3}[90](7,1){$q_{r-1}$}
\end{pspicture}
\end{center}

\noindent
(the double bar between nodes labelled~$q_1$ and~$q_2$ is equivalent to an edge labelled 4).

Whether by an easy adaptation of~\cite[Lemma~1.2 \&~Theorem~1.3]{crisp} or a direct application of~\cite[Proposition 5.4]{dehornoy:embedding}, several embedding criteria can be applied successfully within the current framework. We obtain that the map induced by~$q_1 \mapsto t_i t_{i-1}$ and~$q_j \mapsto s_{j+1}$ for~$j>1$ gives rise to an injection~$\BBc(B_{r-1})\hookrightarrow\BBo(e,e,r)$, hence an injection~$\BB(B_{r-1}) \hookrightarrow\BB(e,e,r)$. This embedding will be used in Subsection~\ref{sss:GarsideElement}. 
\end{enumerate}

\subsection{ The new presentation is~$\BB(e,e,r)$}
\label{ss:isBeerPres}

Our aim in this subsection is to prove the theorem announced in the opening subsection:

\textbf{Theorem~\ref{beerpresthm}.}
\emph{
The presentation~$\PPo(e,e,r)$ is a group presentation for the braid group~$\BB(e,e,r)$.
Furthermore, adding the relations~$a^2=1$ for all generators~$a$ gives a presentation of 
the 
reflection group~$\GG(e,e,r)$. In particular, the generators of this presentation
are all reflections.}

{ We will use the presentation of~\cite{bmr} as our starting point, given on page~\pageref{BMRpres}. 
To this presentation we will add generators~$t_i$ for $2\leq i\leq e$  corresponding to conjugates of~$t_0$ and~$t_1$ which may be defined inductively by:
$$t_i = t_{i-1} t_{i-2} t_{i-1}^{-1} \quad \mbox{ for } i \geq 2.$$
We then verify that the new relations given are both necessary and sufficient. To do this, we introduce the notion of a \emph{circle of elements in a group}, as~$\TE = \{t_i \mid i \in \Z/e\}$ turns out to be the \emph{circle on~$(t_1,t_0)$} in~$\BB(e,e,r)$. }

\subsubsection{ Circles of elements in a group}
\label{sss:circles}

Let~$G$ be a group and~$g_1, g_0$ elements of~$G$. Define elements~$g_i$ for~$i \in \Z$ inductively by:
$$g_i = \left\{ 
\begin{array}{ll}
g_{i-1}^{\phantom{-1}} g_{i-2}^{\phantom{-1}} g_{i-1}^{-1} & \mbox { if } i >1, \mbox{ and } \\
g_{i+1}^{-1}g_{i+2}^{\phantom{-1}} g_{i+1}^{\phantom{-1}}  & \mbox { if } i <0.
\end{array} \right.$$
Then for all~$i, j \in \Z$, the relation
$$g_i g_{i-1} = g_j g_{j-1}$$
is satisfied. The element thus represented is~$g_1 g_0$; denote it by~$\gamma$, and call it the \emph{disk element}.
We call the set~$\{g_i | i \in \Z \}$ the \emph{circle of elements on~$(g_1,g_0)$}, and denote it~$C(g_1,g_0)$. Observe that for any~$i \in Z$, 
$$g_i  = \gamma g_{i-1}^{-1} = g_{i+1}^{-1} \gamma.$$

Conversely, suppose that a group has a set  of elements~$K=\{h_i | i \in \Z \}$ (possibly with 
doubling up, that is, with~$h_i = h_j$ for distinct~$i$ and~$j$)
such  that
$h_i h_{i-1} = h_j h_{j-1}$ for all~$i,j \in \Z$. Then $K$ is~$C(h_p,h_{p-1})$ for any~$p\in\Z$. 

From now on, suppose that $C(g_1,g_0)$ is a circle with disk element~$\gamma$.

\begin{lemma}
\label{tauUpsi}
We have~$\gamma g_i = g_{i+2} \gamma$ for all~$i\in\Z$.
\end{lemma}

\begin{proof}For all~$i \in \Z$, we have $\gamma g_i = g_{i+2} g_{i+1} g_i= g_{i+2} \gamma$.
\end{proof}

In general, the circle of elements obtained may be infinite: for example, in the rank~two free group generated by~$\{g_1, g_0\}$, the circle~$C(g_1,g_0)$ is infinite. Obviously, if the group is finite, then 
any circle of elements is finite. 

\begin{lemma}
\label{closedcircle}
If there exist~$p\in \Z$ and~$e \in \N$ satisfying~$g_p = g_{p+e}$, then we have~$g_i = g_{i+e}$ for all~$i\in \Z$, and~$| C(g_1,g_0)|$ divides~$e$.
\end{lemma}

\begin{proof} 
The proof goes by induction in two directions. Suppose first $q \geq p$ and~$g_j = g_{j+e}$ for all~$j$ with~$p \leq j \leq q$. Then we have
$g_{q+1}^{\phantom{-1}} = \gamma g_q^{-1} = \gamma g_{q+e}^{{-1}}   = g_{q+e+1}^{\phantom{-1}}$,
so the result is true for all~$j \geq p$. Similarly, for~$q \leq p$, $g_j = g_{j+e}$ for all~$j \geq q$ implies
$g_{q-1}^{\phantom{-1}} = g_q^{-1} \gamma = g_{q+e}^{-1} \gamma  = g_{q+e-1}^{\phantom{-1}}.$

Thus $C(g_1,g_0)$ is~$\{ g_i^{\phantom{1}}  | \, i \in \Z/e\}$ and is of cardinality dividing~$e$.
\end{proof}


\begin{lemma}
\label{CoxeterRelationforClosedCircle}
The circle~$C(g_1,g_0)$ is
of finite cardinality  if and only if $\<g_1 g_0\>^e =\<g_0 g_1\>^e$ holds for some~$e \in \N$. 
The smallest~$e$ for which  this relation holds is the cardinality of~$C(g_1,g_0)$.
\end{lemma}

\begin{proof} 
Suppose first~$|C(g_1,g_0)| = e < \infty$. So $g_i = g_{i+e}$ holds for all~$i \in \Z/e$.
Using Lemma~\ref{tauUpsi} above, we find, for~$e$ odd,
$$\<g_1 g_0\>^e = \gamma^{\frac{e-1}{2}} g_1 = 
g_{1+ e-1} \gamma^{\frac{e-1}{2}}  = g_0 \<g_1 g_0\>^{\frac{e-1}{2}}
= \<g_0 g_1\>^e,$$
and, for~$e$ even, 
$$ \<g_0 g_1\>^e = g_0 \gamma^{\frac{e-2}{2}} g_1 = 
g_0 g_{1+e-2}\gamma^{\frac{e-2}{2}} = 
g_0 g_{-1}\gamma^{\frac{e-2}{2}} = \gamma^{\frac{e}{2}} =\<g_1 g_0\>^e.$$

Now assume that~$g_1,g_0$ are elements satisfying~$\<g_1 g_0\>^e = \<g_0 g_1\>^e.$ We show
$$g_q = \<g_1 g_0 \>^q \big( \<g_1 g_0\>^{q-1} \big)^{-1}$$
by induction on~$q>1$. It is certainly true for~$q=1$ and~$q=2$. Suppose~$q \geq 2$. Then we obtain:
$$\begin{array}{rcl}
g_{q+1}^{\phantom{1}}&=&g_q^{\phantom{1}}  g_{q-1}^{\phantom{1}}  g_q^{-1}\\
&= &\<g_1 g_0 \>^q \big( \<g_1 g_0\>^{q-1} \big)^{-1} 
\<g_1 g_0\>^{q-1}  \big( \<g_1 g_0\>^{q-2} \big)^{-1}
 \<g_1 g_0\>^{q-1} \big( \<g_1 g_0\>^{q}  \big)^{-1} \\
&= &\<g_1 g_0 \>^q \big( \<g_1 g_0\>^{q-2} \big)^{-1}
 \<g_1 g_0\>^{q-1} \big( \<g_1 g_0\>^{q}  \big)^{-1} \\
&= &\<g_1 g_0 \>^2 
 \<g_1 g_0\>^{q-1} \big( \<g_1 g_0\>^{q}  \big)^{-1} \ = \ 
 \<g_1 g_0\>^{q+1} \big( \<g_1 g_0\>^{q}  \big)^{-1},
 \end{array}
$$
which concludes the induction. In particular, we find$$g_e = \<g_1 g_0 \>^e \big( \<g_1 g_0\>^{e-1} \big)^{-1} = \<g_0 g_1 \>^e \big( \<g_1 g_0\>^{e-1} \big)^{-1} = g_0,$$ so 
 by Lemma~\ref{closedcircle}, $g_q = g_{q+e}$ holds for all~$q \in \Z$ and~$|C(g_1,g_0)|$ divides~$e$.
\end{proof}

\begin{remark}\label{R:dihedral}
\emph{
The braid group $\BB(I_2(e))$, with reflection group the dihedral group of order $2e$, may be presented by~$\big\<a,b \mid \<ab\>^e = \<ba\>^e \big\>$.
This presentation gives rise to a Garside structure (see Subsection~\ref{ss:basicsGarside} for details about
Garside structures; this fact was proved in~\cite{brsa,del}) corresponding to the classical braid monoid~$\BBc(I_2(e))$.
Lemma~\ref{CoxeterRelationforClosedCircle} implies the known fact that~$\BB(I_2(e))$
also has the presentation
$\big\<a_i, i \in \Z/e \mid a_i a_{i-1} = a_j a_{j-1} \mbox{ for all } i,j \in \Z/e \big\>$, which gives rise to an alternative Garside structure, corresponding to the dual braid monoid~$\BBd(I_2(e))$.
}
\end{remark}

\begin{lemma}
\label{commuters}
Every element~$b$ satisfying~$b g_i = g_i b$ for~$i\in\{0,1\}$ satisfies $b g_i = g_{i} b$ for all~$i \in \Z$.
\end{lemma}

\begin{proof}Clearly, all the elements of~$C(g_1,g_0)$ lie in the subgroup generated by~$g_1$ and~$g_0$. Thus if there is an element which commutes with~$g_1$ and~$g_0$, then it commutes with the entire circle.
\end{proof}

The last property below describes how certain relations on~$g_1$ and~$g_0$
may be extended to the entire circle~$C(g_1,g_0)$. 

\begin{lemma}
\label{gibraidrels}
Every element~$a$ satisfying
$a g_i^{\phantom{1}}Êa = g_i^{\phantom{1}} a g_i^{\phantom{1}}$ 
for~$i\in\{0,1\}$ satisfies\begin{itemize}
\item[$(a)$] $a g_i a = g_i a g_i$ for all~$i \in \Z$, and 
\item[$(b)$] $a \gamma a \gamma = \gamma a \gamma a$.
\end{itemize}
\end{lemma}

\begin{proof}(a) The proof is again by induction in two directions.
We prove the case~$i \geq 1$, the case~$i < 0$ is similar.
Using only the relations of the form~$g_i g_{i-1} = g_j g_{j-1}$ and
those of the form~$g_j a g_j = a g_j a$  for~$0 \leq j \leq i$, we have (see Figure~\ref{BraidRelDiagram})
$$\begin{array}{rcl}
g_{i+1} a g_{i+1}^{-1} 
&=& g_i g_{i-1} g_{i}^{-1} a g_i g_{i-1}^{-1} g_{i}^{-1} \\
&=&  g_i g_{i-1} a g_i a^{-1}   g_{i-1}^{-1} g_{i}^{-1} \\
&=&  g_i g_{i-1} a g_{i-1} g_{i-2} g_{i-1}^{-1} a^{-1}   g_{i-1}^{-1} g_{i}^{-1} \\
&=&  g_i a g_{i-1} a  g_{i-2} a^{-1}  g_{i-1}^{-1} a^{-1}  g_{i}^{-1} \\
&=&  g_i a g_{i-1} g_{i-2}^{-1}a  g_{i-2}   g_{i-1}^{-1} a^{-1}  g_{i}^{-1}    \\
&=& g_i a g_{i}^{-1} g_{i-1}a   g_{i-1}^{-1} g_{i}  a^{-1}  g_{i}^{-1}    \\
&=& a^{-1} g_i a a^{-1} g_{i-1} a   a^{-1}  g_{i}^{-1}  a \qquad 
= \quad a^{-1} g_{i+1} a.
\end{array}
$$

\begin{figure}[h]
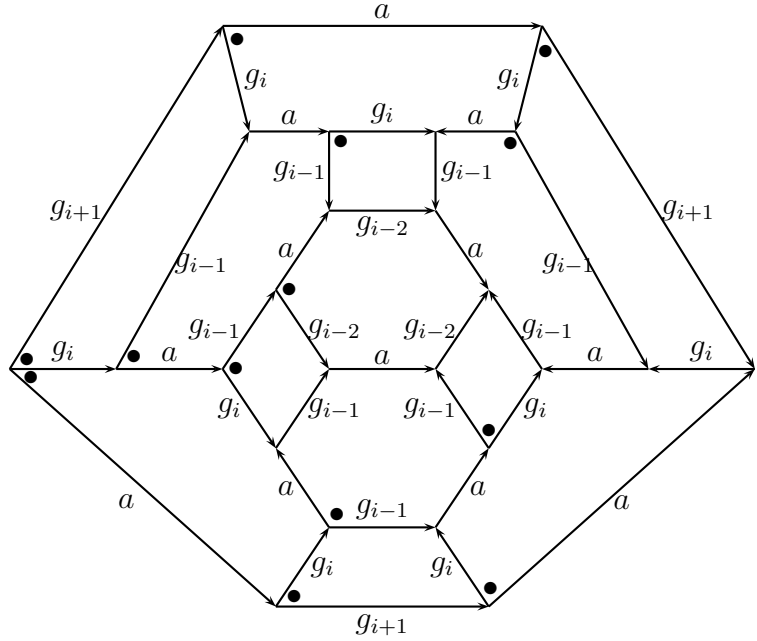

\vskip85mm
\hskip-10cm
\psset{unit=3.5mm}
\uput{0.35}[30](1,10){$\bullet$}
\psline{->}(1,10)(5,10)
\uput{.3}[90](3,10){$g_i$}
\psline{->}(5,10)(10,19)
\uput{.2}[0](7,14){$g_{i-1}$}
\psline{->}(1,10)(9,23)
\uput{.5}[180](5,16){$g_{i+1}$}
\psline{->}(9,23)(10,19)
\uput{.3}[0](9.5,21){$g_{i}$}
\uput{0.35}[315](9,23){$\bullet$}
\psline{->}(9,23)(21,23)
\uput{.3}[90](15,23){$a$}
\psline{->}(10,19)(13,19)
\uput{.3}[90](11.5,19){$a$}
\psline{->}(13,19)(17,19)
\uput{.3}[90](15,19){$g_{i}$}
\psline{->}(20,19)(17,19)
\uput{.3}[90](18.5,19){$a$}
\psline{->}(21,23)(20,19)
\uput{.3}[180](20.5,21){$g_i$}
\uput{0.7}[278](21,23){$\bullet$}
\psline{->}(21,23)(29,10)
\uput{.5}[0](25,16){$g_{i+1}$}
\psline{->}(29,10)(25,10)
\uput{.3}[90](27,10){$g_{i}$}
\psline{->}(20,19)(25,10)
\uput{0}[180](23,14){$g_{i-1}$}
\uput{0.5}[335](1,10){$\bullet$}
\psline{->}(1,10)(11,1)
\uput{.3}[180](6,5){$a$}
\psline{->}(11,1)(13,4)
\uput{.3}[0](12,2.5){$g_{i}$}
\psline{->}(13,4)(11,7)
\uput{.3}[180](12,5.5){$a$}
\psline{->}(5,10)(9,10)
\uput{.3}[90](7,10){$a$}
\psline{->}(9,10)(11,7)
\uput{.3}[180](10,8.5){$g_i$}
\uput{0.4}[30](11,1){$\bullet$}
\psline{->}(11,1)(19,1)
\uput{.3}[270](15,1){$g_{i+1}$}
\psline{->}(19,1)(17,4)
\uput{.3}[180](18,2.5){$g_{i}$}
\psline{->}(13,4)(17,4)
\uput{.3}[90](15,4){$g_{i-1}$}
\uput{0.4}[85](19,1){$\bullet$}
\psline{->}(19,1)(29,10)
\uput{.3}[270](24,5.5){$a$}
\psline{->}(25,10)(21,10)
\uput{.3}[90](23,10){$a$}
\psline{->}(17,4)(19,7)
\uput{.3}[0](18,5.5){$a$}
\psline{->}(19,7)(21,10)
\uput{.3}[0](20,8.5){$g_{i}$}
\uput{0.4}[35](5,10){$\bullet$}
\psline{->}(9,10)(11,13)
\uput{.3}[180](10,11.5){$g_{i-1}$}
\psline{->}(11,13)(13,16)
\uput{.3}[180](12,14.5){$a$}
\psline{->}(13,19)(13,16)
\uput{.1}[180](13,17.5){$g_{i-1}$}
\uput{0.2}[315](13,19){$\bullet$}
\psline{->}(17,19)(17,16)
\uput{.2}[0](17,17.5){$g_{i-1}$}
\psline{->}(13,16)(17,16)
\uput{.2}[270](15,16){$g_{i-2}$}
\uput{0.2}[250](20,19){$\bullet$}
\psline{->}(17,16)(19,13)
\uput{.2}[0](18,14.5){$a$}
\psline{->}(21,10)(19,13)
\uput{.2}[0](20,11.5){$g_{i-1}$}
\uput{0.2}[0](9,10){$\bullet$}
\psline{->}(11,13)(13,10)
\uput{.2}[0](12,11.5){$g_{i-2}$}
\psline{->}(11,7)(13,10)
\uput{.2}[0](12,8.5){$g_{i-1}$}
\uput{0.2}[0](11,13){$\bullet$}
\psline{->}(13,10)(17,10)
\uput{.2}[90](15,10){$a$}
\psline{->}(17,10)(19,13)
\uput{.2}[180](18,11.5){$g_{i-2}$}
\uput{0.2}[60](13,4){$\bullet$}
\psline{->}(19,7)(17,10)
\uput{.2}[180](18,8.5){$g_{i-1}$}
\uput{0.4}[90](19,7){$\bullet$}
\caption{Proof of~$g_{i+1} a g_{i+1} = a g_{i+1} a$. The word~$g_{i+1} a g_{i+1}$ can be read around the top, the word~$a g_{i+1} a$ around the bottom. The interior cells are bounded by words corresponding to relators in the group. The positive relations can be read in opposite directions, starting from the corner of a given cell with the~$\bullet$ symbol.}
\label{BraidRelDiagram}
\end{figure}

\noindent
(b) By the first part we have~$g_2 a g_2 = a g_2 a$. We find (see Figure~\ref{Length4Relation})
$$\begin{array}[t]{rcl}
a \gamma a \gamma 
&=& ag_1g_0ag_1g_0 
= ag_2g_1ag_1g_0
=  ag_2ag_1ag_0 
=g_2 a g_2 g_1 ag_0 \\
&=&g_2 a g_1 g_0 ag_0
= g_2 g_1a g_1 g_0 a 
= g_1 g_0a g_1 g_0 a
= \gamma a \gamma a. 
\end{array}$$
\end{proof}

\begin{figure}[h]
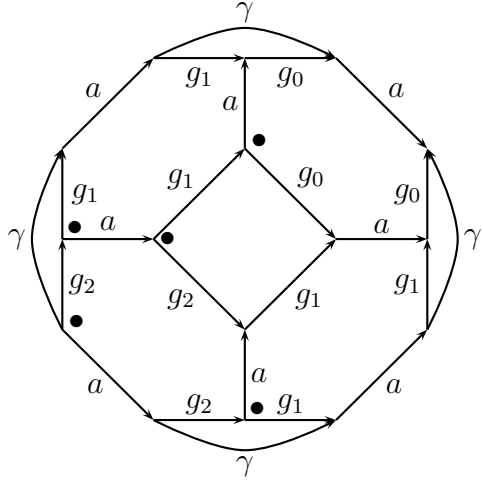

\vskip70mm
\hskip-8cm
\psset{unit=4mm}
\uput{0.2}[45](3,9){$\bullet$}
\psline{->}(3,9)(6,9)
\uput{.3}[90](4.5,9){$a$}
\psline{->}(6,9)(9,12)
\uput{.2}[135](7.5,10.5){$g_1$}
\psline{->}(9,12)(9,15)
\uput{.2}[180](9,13.5){$a$}
\psline{->}(3,9)(3,12)
\uput{.3}[0](3,10.5){$g_1$}
\psline{->}(3,12)(6,15)
\uput{.3}[135](4.5,13.5){$a$}
\psline{->}(6,15)(9,15)
\uput{.3}[270](7.5,15){$g_1$}
\uput{0.2}[30](9,12){$\bullet$}
\psline{->}(9,12)(12,9)
\uput{.3}[45](10.5,10.5){$g_0$}
\psline{->}(12,9)(15,9)
\uput{.2}[90](13.5,9){$a$}
\psline{->}(15,9)(15,12)
\uput{.2}[180](15,10.5){$g_0$}
\psline{->}(9,15)(12,15)
\uput{.3}[270](10.5,15){$g_0$}
\psline{->}(12,15)(15,12)
\uput{.3}[45](13.5,13.5){$a$}
\uput{0.2}[45](9,3){$\bullet$}
\psline{->}(9,3)(9,6)
\uput{.2}[0](9,4.5){$a$}
\psline{->}(9,6)(12,9)
\uput{.2}[315](10.5,7.5){$g_1$}
\psline{->}(9,3)(12,3)
\uput{.2}[90](10.5,3){$g_1$}
\psline{->}(12,3)(15,6)
\uput{.2}[315](13.5,4.5){$a$}
\psline{->}(15,6)(15,9)
\uput{.2}[180](15,7.5){$g_1$}
\uput{0.2}[30](3,6){$\bullet$}
\psline{->}(3,6)(3,9)
\uput{.2}[0](3,7.5){$g_2$}
\psline{->}(6,9)(9,6)
\uput{.2}[225](7.5,7.5){$g_2$}
\psline{->}(3,6)(6,3)
\uput{.2}[225](4.5,4.5){$a$}
\psline{->}(6,3)(9,3)
\uput{.2}[90](7.5,3){$g_2$}
\uput{0.2}[0](6,9){$\bullet$}
\pscurve{->}(3,6)(2,9)(3,12)
\uput{.2}[180](2,9){$\gamma$}
\pscurve{->}(6,15)(9,16)(12,15)
\uput{.2}[90](9,16){$\gamma$}
\pscurve{->}(15,6)(16,9)(15,12)
\uput{.2}[0](16,9){$\gamma$}
\pscurve{->}(6,3)(9,2)(12,3)
\uput{.2}[270](9,2){$\gamma$}
\vskip-5mm
\caption{Proof of~$\gamma a \gamma a = a \gamma a \gamma$. The word~$\gamma a \gamma a$ can be read starting at the left around the top, the word~$a \gamma a \gamma$ around the bottom. The interior cells are bounded by relators.}
\label{Length4Relation}
\end{figure}

\subsubsection{ Proof of Theorem~\ref{beerpresthm}}
\label{sss:theorem1}
We now have enough to prove the theorem.

\emph{Proof of Theorem~\ref{beerpresthm}.}
Let~$J$ denote the group presented by~$\PPo(e,e,r)$.
Relation~$(R_5)$ says that~$\TE$ is the circle~$C(t_1,t_0)$.
By Lemma~\ref{CoxeterRelationforClosedCircle}, $(P_1)$ holds in~$J$. By definition, Relation~$(P_2)$ (\resp $(P_3)$) is a particular case of~$(R_3)$ (\resp $(R_4)$). Relation~$(P_4)$ is precisely a case of Lemma~\ref{gibraidrels}(b) with~$a=s_3$.
Thus all the relations of~$\BB(e,e,r)$ hold in~$J$.

On the other hand, $(P_1)$ says that if $\TT= \{t_i | i \in \Z \} = C(t_1,t_0)$ holds in~$\BB(e,e,r)$, then by Lemma~\ref{CoxeterRelationforClosedCircle},
$t_i = t_{i+e}$ holds for all~$i$, which implies~$(R_5)$. Lemma~\ref{gibraidrels}(a) implies that
$(R_3)$ then holds for all~$t_i \in \TT$.  Lemma~\ref{commuters} ensures that~$(R_4)$ holds.
Thus~$J$ and~$\BB(e,e,r)$ are isomorphic.

The new presentation has the same generators as the original, as well as some conjugates
of the originals. Since it is the case for the presentation in~\cite{bmr}, adding the relations~$a^2=1$ for all generators~$a$ in the new presentation gives a presentation of 
the reflection group~$\GG(e,e,r)$. Since conjugates of reflections are reflections, 
the generators of this presentation are all reflections. Denote by~$\overline{\phantom{x}}$ the natural map~$\BB(e,e,r) \twoheadrightarrow\GG(e,e,r)$. 
The generating reflections in the new presentation of~$\GG(e,e,r)$ are the matrices:
$$\overline{t_i} =\left(\begin{array}{c|c}
 \begin{array}{rl}
0 &  \zeta_e^{-i} \\
 \zeta_e^i &  0 \\
\end{array}
& \begin{array}{c} \\ {\LARGE 0} \\ \\ \end{array} \\
\hline 
\\
0  & ~ {\LARGE I_{r-2}} ~ \\
\\
\end{array} \right)
\quad \mbox{ and } \quad
\overline{s_j} = \mbox{ matrix of } (j-1 \ \ \ j),$$
where~$\zeta_e$ is a primitive $e$-th root of unity.
\qed

\section{ A new Garside structure}
\label{s:newGarsideStructure}

In this section, we first give a proof that the monoid~$\BBo(e,e,r)$---defined by the presentation studied in the previous section---is a Garside monoid. Then we find a precise description of the combinatorics of the underlying Garside structure. Finally we produce some arguments in order to convince the reader that this structure could be named post-classical.

\subsection{Background on Garside theory}
\label{ss:basicsGarside}

In this preliminary subsection, we list some basic definitions and summarize results by~Dehornoy \&~Paris about Garside theory. For all the results quoted here, we refer the reader
to~\cite{dfx,dehornoy:embedding,dehornoy:garside,dehornoy:complete}.

{ For~$x,y$ in a monoid~$M$, write~$x \cleq y$ if there exists~$z \in M$ satisfying~$xz = y$, and say either that~$x$ left-divides~$y$ or that~$y$ is a right multiple of~$x$. 

There are similar definitions for right division and left multiplication (with the  notation~$y \cgeq x$ if there exists~$z$ satisfying~$y = zx$).
Write~$x \vee y$ for the right lcm of~$x$ and~$y$, and write~$x \wedge y$ for the left gcd. When~$M$ is cancellative, elements~$(x\backslash y)$ and~$(y\backslash x)$ are uniquely defined by:
$$x\vee y=x(x\backslash y)=y(y\backslash x).$$

A \emph{Garside monoid}~$M$ is a cancellative monoid with lcm's and gcd's and admitting a \emph{Garside element}, namely an element whose left and right divisors coincide, are finite in number\footnote{Finiteness is a quite technical condition which can be relaxed in some contexts.} and generate~$M$. There exists a minimal Garside element---usually denoted by~$\Delta$ and then called \emph{the} Garside element---whose divisors are called \emph{the simples} of~$M$.

By \"Ore's conditions, a Garside monoid embeds in a group of fractions. A \emph{Garside group} is a group that is the group of fractions of (at least) one Garside monoid.

Recognizing a Garside monoid from a presentation and computing in a Garside group given by a presentation are natural questions which can be solved by using \emph{word reversing}, a syntactic method relevant for semigroup presentations.

Let~$\varepsilon$ denote the empty word. For~$\langle~\A\mid\RR~\rangle$ a semigroup presentation and~$w,w'$ words on~$\A\cup\A^{-1}$, we say that \emph{$w$ reverses to~$w'$}---written~$w \revstorelR w'$---if~$w'$ is obtained from~$w$ by (iteratively)\begin{itemize}
\item deleting some~$x^{-1}x$ for~$x\in\A$,
\item replacing some~$x^{-1}y$ with~$uv^{-1}$ for~$xu=yv$ a relation in~$\RR$.
\end{itemize}
This can be represented diagrammatically as shown in Figure~\ref{wrDiag}.

\begin{figure}[h]
\begin{center}
\vskip5mm
\psset{unit=8mm}
\psline{->}(0,0)(0,-2)
\uput{.2}[180](0,-1){$x$}
\psline{->}(0,0)(2,0)
\uput{.2}[90](1,0){$x$}
\psarc(0,0){2}{270}{0}
\uput{0}[0](1.5,-1.5){$\varepsilon$}
\uput{0}[0](0.5,-0.8){$\revsto$}
\hskip5cm
\psline{->}(0,0)(0,-2)
\uput{.2}[180](0,-1){$x$}
\psline{->}(0,0)(2,0)
\uput{.2}[90](1,0){$y$}
\uput{0}[0](0.8,-1){$\revsto$}
\psline{->}(0,-2)(2,-2)
\uput{.2}[270](1,-2){$u$}
\psline{->}(2,0)(2,-2)
\uput{.2}[0](2,-1){$v$}
\vskip18mm
\end{center}
\caption{Word reversing diagrams for~$x^{-1}x \revsto \varepsilon$ and~$x^{-1} y \revsto uv^{-1}$.}
\label{wrDiag}
\end{figure}

First, remark that, for all~$u,v\in\A^*$, $u^{-1}v\revstorelR\varepsilon$ implies~$u\equiv^+_\RR v$, where~$\equiv^+_\RR$ denotes the  monoid congruence
generated by~$\RR$. A semigroup presentation~$\langle~\A\mid\RR~\rangle$ is said to be \emph{complete (for reversing)} when the converse holds, that is, when word reversing detects equivalence.

Technically, $\langle~\A\mid\RR~\rangle$ is complete if and only if every triple~$(u,v,w)$ of words over~$\A$ satisfies the \emph{cube condition~(CC) modulo~$\RR$}:
\begin{center}
\fbox{\parbox{288pt}{
\vspace{1mm}{
$u^{-1} w w^{-1} v \revstorelR u v^{-1}$ 
for~$u,v\in\A^*$ implies $(xu)^{-1} yv \revstorelR \varepsilon$.}
\vspace{-1mm}}}
\end{center}
The CC can be represented diagrammatically as shown in Figure~\ref{ccDiag}.

\begin{figure}[h]
\vskip30mm
\hskip-8cm
\psset{unit=6mm}
\psline{->}(0,2)(0,0)
\uput{.2}[180](0,1){$u$}
\psline{->}(0,2)(2,2)
\uput{.2}[90](1,2){$w$}
\psline{->}(2,4)(2,2)
\uput{.2}[180](2,3){$w$}
\psline{->}(2,4)(4,4)
\uput{.2}[90](3,4){$v$}
\uput{0}[0](2.5,1.5){$\revsto$}
\psline{->}(0,0)(4,0)
\uput{.2}[270](2,0){$v'$}
\psline{->}(4,4)(4,0)
\uput{.2}[0](4,2){$u'$}
\uput{0}[0](6,2){$\Longrightarrow$}
\psline{->}(10,4)(10,0)
\uput{.2}[180](10,2){$uv'$}
\psline{->}(10,4)(14,4)
\uput{.2}[90](12,4){$vu'$}
\uput{0}[0](11.5,2){$\revsto$}
\psarc(10,4){4}{270}{0}
\uput{0}[0](13,1){$\varepsilon$}
\vskip3mm
\caption{Cube condition: $u^{-1}w w^{-1} v \revsto v'u'^{-1} \Rightarrow (uv')^{-1} vu' \revsto \varepsilon$.}
\label{ccDiag}
\end{figure}

In the general case, the cube condition has to be checked for all triples of words on~$\A$, or for all triples of words in a superset of~$\A$ closed under~$\revsto$. However, in the homogeneous case (that is, when every relation preserves the length of words), check the cube condition for all triples of generators suffices to decide completeness.

A semigroup presentation~$\langle~\A\mid\RR~\rangle$ is \emph{complemented} if, for all generators~$x,y$ in~$\A$, there is at most one relation of the type~$x\cdots=y\cdots$ and no relation of the type~$x\cdots=x\cdots$. We will use the following criterium:

\begin{theorem}\cite{dehornoy:complete}\label{garsideCriterium} Every monoid defined by a complemented complete presentation and admitting a Garside element is a Garside monoid.
\end{theorem}}

\subsection{ The braid monoid~$\BBo(e,e,r)$ is Garside}
\label{ss:presIsGarside}

The aim of this subsection is to show:

\begin{theorem} 
\label{th:garside}
The braid monoid~$\BBo(e,e,r)$ is Garside.
\end{theorem}

We prove the theorem in two parts: first completeness, then the Garside element.

\subsubsection{ Completeness}
\label{sss:completeness}

Taking advantage from knowledge of completeness for presentations associated to certain parabolic subdiagrams, we can check completeness of~$\PPo(e,e,r)$ after computing only few cases.
 
\begin{lemma}\label{complete}The presentation~$\PPo(e,e,r)$ is complemented and complete.
\end{lemma}

\begin{proof}
The presentation~$\PPo(e,e,r)$ is complemented and homogeneous. Now, it suffices to check whether every triple~$(x,y,z)$ of generators in~$\TE\cup\SS$ satisfies the~CC.

From Subsection~\ref{sss:parabolic} about parabolic subdiagrams, we deduce:

\begin{enumerate}
\item[$(\mathbf{0+3})$] Every triple in~$\SS^3$ satisfies the CC because it holds for~$\BBc(A_{r-2})$.
%
\item[$(\mathbf{1+2})$] Every triple in~$(\TE\times\SS^2)\cup(\SS\times\TE\times\SS)\cup(\SS^2\times\TE)$ satisfies the CC because it holds for~$\BBc(A_{r-1})$.
\item[$(\mathbf{3+0})$] Every triple in~$\TE^3$ satisfies the CC, because it holds for~$\BBd(I_2(e))$.
\end{enumerate}

Thus we need only verify the cube condition on triples of type~$(\mathbf{2+1})$, that is, containing two generators from~$\TE$ and one generator from~$\SS$. This case can be decomposed into two subcases depending on whether this generator from~$\SS$ is~$s_3$ (say case~$(\mathbf{a})$) or not (case~$(\mathbf{b})$). From Subsection~\ref{sss:parabolic} again, with~$\SSM=\SS\setminus\{s_3\}$, we find:
\begin{enumerate}
\item[$(\mathbf{2+1b})$] Every triple in~$(\TE^2\times\SSM)\cup(\TE\times\SSM\times\TE)\cup(\SSM\times\TE^2)$ satisfies the CC because it holds for~$\BBd(I_2(e))\times\BBc(A_{r-3})$.
\end{enumerate}
Therefore we need only verify the CC on triples of type~$(\mathbf{2+1a})$, that is, containing two generators from~$\TE$ and the generator~$s_3$ from~$\SS$. Now, various symmetry considerations reduce again the number of cases that need be considered.

\medbreak

\vbox{On the one hand, triples of the form~$(x,y,x)$ and~$(x,x,y)$ always satisfy the CC.
\vskip20mm
\hskip45mm
\psset{unit=5mm}
\uput{0}[0](-8,2){Case~$(x,y,x)$:}
\psline{->}(0,2)(0,0)
\uput{.2}[180](0,1){$x$}
\psline{->}(0,2)(2,2)
\uput{.2}[90](1,2){$x$}
\psline{->}(2,4)(2,2)
\uput{.2}[180](2,3){$x$}
\psline{->}(2,4)(4,4)
\uput{.2}[90](3,4){$y$}
\psarc(0,2){2}{270}{0}
\uput{0}[0](2.8,3){$\revsto$}
\psline{->}(2,2)(4,2)
\uput{.2}[270](3,2){$u$}
\psline{->}(4,4)(4,2)
\uput{.2}[0](4,3){$v$}
\uput{0}[0](6,2){$\Longrightarrow$}
\psline{->}(10,4)(10,2)
\uput{.2}[180](10,3){$x$}
\psline{->}(10,2)(10,0)
\uput{.2}[180](10,1){$u$}
\psline{->}(10,4)(12,4)
\uput{.2}[90](11,4){$y$}
\psline{->}(12,4)(14,4)
\uput{.2}[90](13,4){$v$}
\uput{0}[0](10.8,3){$\revsto$}
\psline{->}(10,2)(12,2)
\uput{.2}[270](11,2){$u$}
\psline{->}(12,4)(12,2)
\uput{.2}[0](12,3){$v$}
\psarc(10,2){2}{270}{0}
\psarc(12,4){2}{270}{0}
\uput{0}[0](11.5,0.5){$\varepsilon$}
\uput{0}[0](13.5,2.5){$\varepsilon$}
\vskip3mm

\vskip20mm
\hskip45mm
\psset{unit=5mm}
\uput{0}[0](-8,2){Case~$(x,x,y)$:}
\psline{->}(0,2)(0,0)
\uput{.2}[180](0,1){$x$}
\psline{->}(0,2)(2,2)
\uput{.2}[90](1,2){$y$}
\psline{->}(2,4)(2,2)
\uput{.2}[180](2,3){$y$}
\psline{->}(2,4)(4,4)
\uput{.2}[90](3,4){$x$}
\uput{0}[0](0.8,1){$\revsto$}
\psline{->}(0,0)(2,0)
\uput{.2}[270](1,0){$u$}
\psline{->}(2,2)(2,0)
\uput{.2}[0](2,1){$v$}
\uput{0}[0](2.8,3){$\revsto$}
\psline{->}(2,2)(4,2)
\uput{.2}[270](3,2){$v$}
\psline{->}(4,4)(4,2)
\uput{.2}[0](4,3){$u$}
\psarc(2,2){2}{270}{0}
\uput{0}[0](6,2){$\Longrightarrow$}
\psline{->}(10,4)(10,2)
\uput{.2}[180](10,3){$x$}
\psline{->}(10,2)(10,0)
\uput{.2}[180](10,1){$u$}
\psline{->}(10,4)(12,4)
\uput{.2}[90](11,4){$x$}
\psline{->}(12,4)(14,4)
\uput{.2}[90](13,4){$u$}
\psarc(10,4){2}{270}{0}
\psarc(10,4){4}{270}{0}
\uput{0}[0](13,1){$\varepsilon$}
\vskip3mm}
%
On the other hand, a triple~$(x,y,z)$ satisfies the CC if and only if~$(y,x,z)$ satisfies the CC: this may be seen by reflecting the word reversing diagram through an axis at~$\frac{3\pi}{4}$.

This results in the following two cases:
\begin{itemize}
\item case~$(x,y,z) = (t_i, t_j, s_3)$ with~$i$ and~$j$ distinct;
\item case~$(x,y,z) = (t_i,  s_3, t_j)$ with~$i$ and~$j$ distinct.
\end{itemize}
The calculations are shown in Figures~\ref{C:case001b} and \ref{C:case001a} respectively.\end{proof}

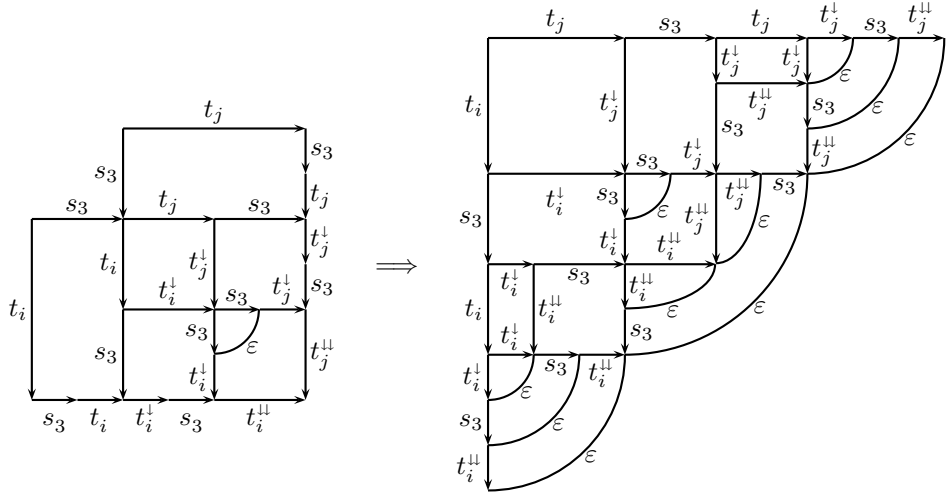
\begin{figure}[h]
	\begin{center}
	\hspace*{-20pt}
	\begin{pspicture*}(1.6,1)(15,8)
	\psset{unit=12mm}
	\footnotesize
		\psline{->}(3,5)(5,5)
		\uput{0.05}[90](4,5){$t_j$}
		\psline{->}(2,4)(3,4)
		\uput{0.05}[90](2.5,4){$s_3$}
		\psline{->}(3,4)(4,4)
		\uput{0.05}[90](3.5,4){$t_j$}
		\psline{->}(4,4)(5,4)
		\uput{0.05}[90](4.5,4){$s_3$}
		\psline{->}(2,4)(2,2)
		\uput{0.05}[180](2,3){$t_i$}
		\psline{->}(3,5)(3,4)
		\uput{0.05}[180](3,4.5){$s_3$}
		\psline{->}(3,4)(3,3)
		\uput{0.05}[180](3,3.5){$t_i$}
		\psline{->}(3,3)(3,2)
		\uput{0.05}[180](3,2.5){$s_3$}
		\psline{->}(4,4)(4,3)
		\uput{0.05}[180](4,3.5){$t_j^{\ssda}$}
		\psline{->}(4,3)(4,2.5)
		\uput{0.05}[180](4,2.75){$s_3$}
		\psline{->}(4,2.5)(4,2)
		\uput{0.05}[180](4,2.25){$t_i^{\ssda}$}
		\psline{->}(3,3)(4,3)
		\uput{0.05}[90](3.5,3){$t_i^{\ssda}$}
		\psline{->}(4,3)(4.5,3)
		\uput{0.05}[90](4.25,3){$s_3$}
		\psline{->}(4.5,3)(5,3)
		\uput{0.05}[90](4.75,3){$t_j^{\ssda}$}
		\psellipticarc(4,3)(.5,.5){270}{0}
		\uput{0.5}[315](4,3){$\varepsilon$}
		\psline{->}(5,5)(5,4.5)
		\uput{0.05}[0](5,4.75){$s_3$}
		\psline{->}(5,4.5)(5,4)
		\uput{0.05}[0](5,4.25){$t_j$}
		\psline{->}(5,4)(5,3.5)
		\uput{0.05}[0](5,3.75){$t_j^{\ssda}$}
		\psline{->}(5,3.5)(5,3)
		\uput{0.05}[0](5,3.25){$s_3$}
		\psline{->}(5,3)(5,2)
		\uput{0.05}[0](5,2.5){$t_j^{\ssda\!\ssda}$}
		\psline{->}(2,2)(2.5,2)
		\uput{0.05}[270](2.25,2){$s_3^{\phantom{\ssda}}$}
		\psline{->}(2.5,2)(3,2)
		\uput{0.05}[270](2.75,2){$t_i^{\phantom{\ssda}}$}
		\psline{->}(3,2)(3.5,2)
		\uput{0.05}[270](3.25,2){$t_i^{\ssda}$}
		\psline{->}(3.5,2)(4,2)
		\uput{0.05}[270](3.75,2){$s_3^{\phantom{\ssda}}$}
		\psline{->}(4,2)(5,2)
		\uput{0.05}[270](4.5,2){$t_i^{\ssda\!\ssda}$}
		\uput{0}[90](6,3.4){$\Longrightarrow$}
		\psline{->}(7,6)(8.5,6)
		\uput{0.05}[90](7.75,6){$t_j$}
		\psline{->}(8.5,6)(9.5,6)
		\uput{0.05}[90](9,6){$s_3$}
		\psline{->}(9.5,6)(10.5,6)
		\uput{0.05}[90](10,6){$t_j$}
		\psline{->}(10.5,6)(11,6)
		\uput{0.05}[90](10.75,6){$t_j^{\ssda}$}
		\psline{->}(11,6)(11.5,6)
		\uput{0.05}[90](11.25,6){$s_3$}
		\psline{->}(11.5,6)(12,6)
		\uput{0.05}[90](11.75,6){$t_j^{\ssda\!\ssda}$}
		\psline{->}(7,6)(7,4.5)
		\uput{0.05}[180](7,5.25){$t_i$}
		\psline{->}(7,4.5)(7,3.5)
		\uput{0.05}[180](7,4){$s_3$}
		\psline{->}(7,3.5)(7,2.5)
		\uput{0.05}[180](7,3){$t_i$}
		\psline{->}(7,2.5)(7,2)
		\uput{0.05}[180](7,2.25){$t_i^{\ssda}$}
		\psline{->}(7,2)(7,1.5)
		\uput{0.05}[180](7,1.75){$s_3$}
		\psline{->}(7,1.5)(7,1)
		\uput{0.05}[180](7,1.25){$t_i^{\ssda\!\ssda}$}
		\psline{->}(9.5,5.5)(10.5,5.5)
		\uput{0.05}[270](10,5.5){$t_j^{\ssda\!\ssda}$}
		\psline{->}(7,4.5)(8.5,4.5)
		\uput{0.05}[90](7.75,4){$t_i^{\ssda}$}
		\psline{->}(8.5,4.5)(9,4.5)
		\uput{0.05}[90](8.75,4.5){$s_3$}
		\psline{->}(9,4.5)(9.5,4.5)
		\uput{0.05}[90](9.25,4.5){$t_j^{\ssda}$}
		\psline{->}(9.5,4.5)(10,4.5)
		\uput{0.05}[270](9.75,4.5){$t_j^{\ssda\!\ssda}$}
		\psline{->}(10,4.5)(10.5,4.5)
		\uput{0.05}[270](10.25,4.5){$s_3$}
		\psline{->}(7.5,3.5)(7.5,2.5)
		\uput{0.05}[0](7.5,3){$t_i^{\ssda\!\ssda}$}
		\psline{->}(8.5,6)(8.5,4.5)
		\uput{0.05}[180](8.5,5.25){$t_j^{\ssda}$}
		\psline{->}(8.5,4.5)(8.5,4)
		\uput{0.05}[180](8.5,4.25){$s_3$}
		\psline{->}(8.5,4)(8.5,3.5)
		\uput{0.05}[180](8.5,3.75){$t_i^{\ssda}$}
		\psline{->}(8.5,3.5)(8.5,3)
		\uput{0.05}[0](8.5,3.25){$t_i^{\ssda\!\ssda}$}
		\psline{->}(8.5,3)(8.5,2.5)
		\uput{0.05}[0](8.5,2.75){$s_3$}
		\psarc(8.5,4.5){0.5}{270}{0}
		\uput{0.5}[315](8.5,4.5){$\varepsilon$}
		\psarc(8.5,4.5){2}{270}{0}
		\uput{2}[315](8.5,4.5){$\varepsilon$}
		\psellipticarc(8.5,3.5)(1,0.5){270}{0}
		\uput{0.65}[315](8.5,3.5){$\varepsilon$}
		\psline{->}(7,3.5)(7.5,3.5)
		\uput{0.05}[270](7.25,3.5){$t_i^{\ssda}$}
		\psline{->}(7.5,3.5)(8.5,3.5)
		\uput{0.05}[270](8,3.5){$s_3$}
		\psline{->}(8.5,3.5)(9.5,3.5)
		\uput{0.05}[90](9,3.5){$t_i^{\ssda\!\ssda}$}
		\psline{->}(9.5,6)(9.5,5.5)
		\uput{0.05}[0](9.5,5.75){$t_j^{\ssda}$}
		\psline{->}(9.5,5.5)(9.5,4.5)
		\uput{0.05}[0](9.5,5){$s_3$}
		\psline{->}(9.5,4.5)(9.5,3.5)
		\uput{0.05}[180](9.5,4){$t_j^{\ssda\!\ssda}$}
		\psellipticarc(9.5,4.5)(0.5,1){270}{0}
		\uput{0.65}[315](9.5,4.5){$\varepsilon$}
		\psline{->}(10.5,6)(10.5,5.5)
		\uput{0.05}[180](10.5,5.75){$t_j^{\ssda}$}
		\psline{->}(10.5,5.5)(10.5,5)
		\uput{0.05}[0](10.5,5.25){$s_3$}
		\psline{->}(10.5,5)(10.5,4.5)
		\uput{0.05}[0](10.5,4.75){$t_j^{\ssda\!\ssda}$}
		\psarc(10.5,6){0.5}{270}{0}
		\uput{0.5}[315](10.5,6){$\varepsilon$}
		\psarc(10.5,6){1}{270}{0}
		\uput{1}[315](10.5,6){$\varepsilon$}
		\psarc(10.5,6){1.5}{270}{0}
		\uput{1.5}[315](10.5,6){$\varepsilon$}
		\psline{->}(7,2.5)(7.5,2.5)
		\uput{0.05}[90](7.25,2.5){$t_i^{\ssda}$}
		\psline{->}(7.5,2.5)(8,2.5)
		\uput{0.05}[270](7.75,2.5){$s_3$}
		\psline{->}(8,2.5)(8.5,2.5)
		\uput{0.05}[270](8.25,2.5){$t_i^{\ssda\!\ssda}$}
		\psarc(7,2.5){0.5}{270}{0}
		\uput{0.5}[315](7,2.5){$\varepsilon$}
		\psarc(7,2.5){1}{270}{0}
		\uput{1}[315](7,2.5){$\varepsilon$}
		\psarc(7,2.5){1.5}{270}{0}
		\uput{1.5}[315](7,2.5){$\varepsilon$}
		\end{pspicture*}
	\end{center}
\caption{\hbox{Proof of~completeness: the case~$(x,y,z) = (t_i, t_j, s_3)$ with~$i\neq j$.}}
\label{C:case001b}
\end{figure}

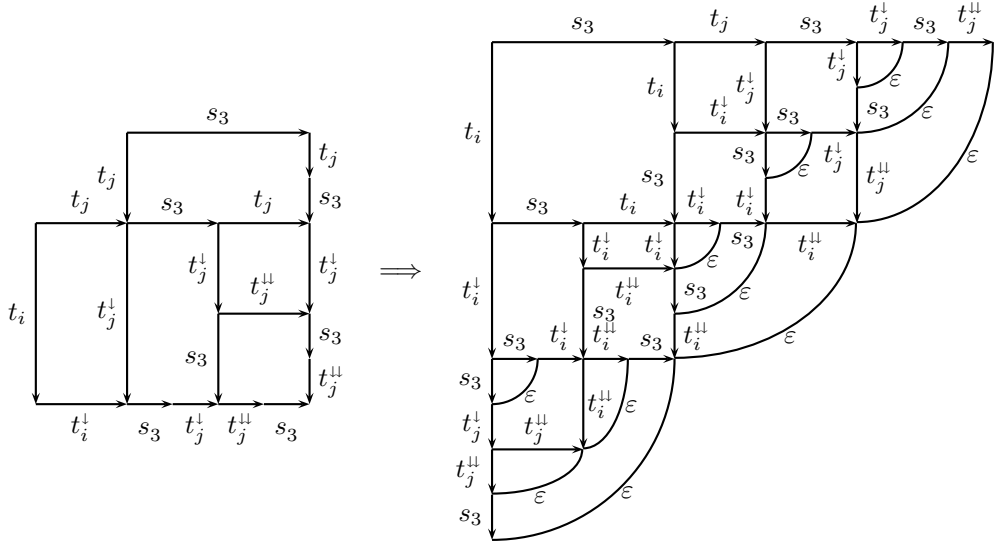
\begin{figure}[h]
	\begin{center}
	\hspace*{-20pt}
	\begin{pspicture*}(1.6,0.5)(15,8)
	\psset{unit=12mm}
		\footnotesize
		\psline{->}(3,5)(5,5)
		\uput{0.1}[90](4,5){$s_3$}
		\psline{->}(2,4)(3,4)
		\uput{0.1}[90](2.5,4){$t_j$}
		\psline{->}(3,4)(4,4)
		\uput{0.1}[90](3.5,4){$s_3$}
		\psline{->}(4,4)(5,4)
		\uput{0.1}[90](4.5,4){$t_j$}
		\psline{->}(2,4)(2,2)
		\uput{0.1}[180](2,3){$t_i$}
		\psline{->}(3,5)(3,4)
		\uput{0.1}[180](3,4.5){$t_j$}
		\psline{->}(3,4)(3,2)
		\uput{0.1}[180](3,3){$t_j^{\ssda}$}
		\psline{->}(4,4)(4,3)
		\uput{0.1}[180](4,3.5){$t_j^{\ssda}$}
		\psline{->}(4,3)(4,2)
		\uput{0.1}[180](4,2.5){$s_3$}
		\psline{->}(4,3)(5,3)
		\uput{0.1}[90](4.5,3){$t_j^{\ssda\!\ssda}$}
		\psline{->}(5,5)(5,4.5)
		\uput{0.1}[0](5,4.75){$t_j$}
		\psline{->}(5,4.5)(5,4)
		\uput{0.1}[0](5,4.25){$s_3$}
		\psline{->}(5,4)(5,3)
		\uput{0.1}[0](5,3.5){$t_j^{\ssda}$}
		\psline{->}(5,3)(5,2.5)
		\uput{0.1}[0](5,2.75){$s_3$}
		\psline{->}(5,2.5)(5,2)
		\uput{0.1}[0](5,2.25){$t_j^{\ssda\!\ssda}$}
		\psline{->}(2,2)(3,2)
		\uput{0.1}[270](2.5,2){$t_i^{\ssda}$}
		\psline{->}(3,2)(3.5,2)
		\uput{0.1}[270](3.25,2){$s_3^{\phantom{\ssda}}$}
		\psline{->}(3.5,2)(4,2)
		\uput{0.1}[270](3.75,2){$t_j^{\ssda}$}
		\psline{->}(4,2)(4.5,2)
		\uput{0.1}[270](4.25,2){$t_j^{\ssda\!\ssda}$}
		\psline{->}(4.5,2)(5,2)
		\uput{0.1}[270](4.75,2){$s_3^{\phantom{\ssda}}$}
		\uput{0}[90](6,3.4){$\Longrightarrow$}
		\psline{->}(7,6)(9,6)
		\uput{0.1}[90](8,6){$s_3$}
		\psline{->}(9,6)(10,6)
		\uput{0.1}[90](9.5,6){$t_j$}
		\psline{->}(10,6)(11,6)
		\uput{0.1}[90](10.5,6){$s_3$}
		\psline{->}(11,6)(11.5,6)
		\uput{0.1}[90](11.25,6){$t_j^{\ssda}$}
		\psline{->}(11.5,6)(12,6)
		\uput{0.1}[90](11.75,6){$s_3$}
		\psline{->}(12,6)(12.5,6)
		\uput{0.1}[90](12.25,6){$t_j^{\ssda\!\ssda}$}
		\psline{->}(7,6)(7,4)
		\uput{0.1}[180](7,5){$t_i$}
		\psline{->}(7,4)(7,2.5)
		\uput{0.1}[180](7,3.25){$t_i^{\ssda}$}
		\psline{->}(7,2.5)(7,2)
		\uput{0.1}[180](7,2.25){$s_3$}
		\psline{->}(7,2)(7,1.5)
		\uput{0.1}[180](7,1.75){$t_j^{\ssda}$}
		\psline{->}(7,1.5)(7,1)
		\uput{0.1}[180](7,1.25){$t_j^{\ssda\!\ssda}$}
		\psline{->}(7,1)(7,0.5)
		\uput{0.1}[180](7,0.75){$s_3$}
		\psline{->}(9,6)(9,5)
		\uput{0.1}[180](9,5.5){$t_i$}
		\psline{->}(9,5)(9,4)
		\uput{0.1}[180](9,4.5){$s_3$}
		\psline{->}(9,4)(9,3.5)
		\uput{0.1}[180](9,3.75){$t_i^{\ssda}$}
		\psline{->}(9,3.5)(9,3)
		\uput{0.1}[0](9,3.25){$s_3$}
		\psline{->}(9,3)(9,2.5)
		\uput{0.1}[0](9,2.75){$t_i^{\ssda\!\ssda}$}
		\psarc(9,4){0.5}{270}{0}
		\uput{0.5}[315](9,4){$\varepsilon$}
		\psarc(9,4){1}{270}{0}
		\uput{1}[315](9,4){$\varepsilon$}
		\psellipticarc(9,4)(2,1.5){270}{0}
		\uput{1.7}[315](9,4){$\varepsilon$}
		\psline{->}(10,6)(10,5)
		\uput{0.1}[180](10,5.5){$t_j^{\ssda}$}
		\psline{->}(10,5)(10,4.5)
		\uput{0.1}[180](10,4.75){$s_3$}
		\psline{->}(10,4.5)(10,4)
		\uput{0.1}[180](10,4.25){$t_i^{\ssda}$}
		\psarc(10,5){0.5}{270}{0}
		\uput{.5}[315](10,5){$\varepsilon$}
		\psline{->}(9,5)(10,5)
		\uput{0.1}[90](9.5,5){$t_i^{\ssda}$}
		\psline{->}(10,5)(10.5,5)
		\uput{0.1}[90](10.25,5){$s_3$}
		\psline{->}(10.5,5)(11,5)
		\uput{0.1}[270](10.75,5){$t_j^{\ssda}$}
		\psline{->}(7,4)(8,4)
		\uput{0.1}[90](7.5,4){$s_3$}
		\psline{->}(8,4)(9,4)
		\uput{0.1}[90](8.5,4){$t_i$}
		\psline{->}(9,4)(9.5,4)
		\uput{0.1}[90](9.25,4){$t_i^{\ssda}$}
		\psline{->}(9.5,4)(10,4)
		\uput{0.1}[270](9.75,4){$s_3$}
		\psline{->}(10,4)(11,4)
		\uput{0.1}[270](10.5,4){$t_i^{\ssda\!\ssda}$}
		\psline{->}(11,6)(11,5.5)
		\uput{0.1}[180](11,5.75){$t_j^{\ssda}$}
		\psline{->}(11,5.5)(11,5)
		\uput{0.1}[0](11,5.25){$s_3$}
		\psline{->}(11,5)(11,4)
		\uput{0.1}[0](11,4.5){$t_j^{\ssda\!\ssda}$}
		\psarc(11,6){0.5}{270}{0}
		\uput{.5}[315](11,6){$\varepsilon$}
		\psarc(11,6){1}{270}{0}
		\uput{1}[315](11,6){$\varepsilon$}
		\psellipticarc(11,6)(1.5,2){270}{0}
		\uput{1.7}[315](11,6){$\varepsilon$}
		\psline{->}(8,4)(8,3.5)
		\uput{0.1}[0](8,3.75){$t_i^{\ssda}$}
		\psline{->}(8,3.5)(8,2.5)
		\uput{0.1}[0](8,3){$s_3$}
		\psline{->}(8,2.5)(8,1.5)
		\uput{0.05}[0](8,2){$t_i^{\ssda\!\ssda}$}
		\psline{->}(8,3.5)(9,3.5)
		\uput{0.1}[270](8.5,3.5){$t_i^{\ssda\!\ssda}$}
		\psline{->}(7,2.5)(7.5,2.5)
		\uput{0.1}[90](7.25,2.5){$s_3$}
		\psline{->}(7.5,2.5)(8,2.5)
		\uput{0.1}[90](7.75,2.5){$t_i^{\ssda}$}
		\psline{->}(8,2.5)(8.5,2.5)
		\uput{0.1}[90](8.25,2.5){$t_i^{\ssda\!\ssda}$}
		\psline{->}(8.5,2.5)(9,2.5)
		\uput{0.1}[90](8.75,2.5){$s_3$}
		\psarc(7,2.5){0.5}{270}{0}
		\uput{0.5}[315](7,2.5){$\varepsilon$}
		\psarc(7,2.5){2}{270}{0}
		\uput{2}[315](7,2.5){$\varepsilon$}
		\psellipticarc(8,2.5)(0.5,1){270}{0}
		\uput{0.65}[315](8,2.5){$\varepsilon$}
		\psline{->}(7,1.5)(8,1.5)
		\uput{0.05}[90](7.5,1.5){$t_j^{\ssda\!\ssda}$}
		\psellipticarc(7,1.5)(1,0.5){270}{0}
		\uput{0.65}[315](7,1.5){$\varepsilon$}
		\end{pspicture*}
	\end{center}
\caption{\hbox{Proof of~completeness: the case~$(x,y,z) = (t_i, s_3, t_j)$ with~$i\neq j$.}}
\label{C:case001a}
\end{figure}
 
\subsubsection{ Garside element}
\label{sss:GarsideElement}

Let~$\tau$ be the element~$t_i t_{i-1}$ of~$\BBo(e,e,r)$. Since $t_i t_{i-1} = t_j t_{j-1}$ holds for all~$i, j$,
$\tau$ is independent of~$i$. It is a common multiple of~$\TE$; and since no word of length one
could be a multiple of all the~$t_i$, $\tau$ is the lcm of~$\TE$.

\medskip

\vbox{
The classical braid monoid~$\BBc(2,1,n)$ for~$\BB(2,1,n)$ (indeed, $\BB(d,1,n)$ for any~$d \geq 2$) is defined by the following Coxeter diagram:

\begin{center}
\begin{pspicture}(0,1)(7,1)
%
\qline(0,.94)(1.6,.94)
\qline(0,1.06)(1.6,1.06)
\qline(1.6,1)(3.9,1)
\psset{linestyle=dashed, dash=1pt 3pt, linecolor=black}
\qline(3.9,1)(4.7,1)
\psset{linestyle=solid}
\qline(4.7,1)(7,1)
\pscircle*[linecolor=white](0,1){.18}
\pscircle(0,1){.18}
\uput{.3}[90](0,1){$q_1$}
\pscircle*[linecolor=white](1.6,1){.18}
\pscircle(1.6,1){.18}
\uput{.3}[90](1.6,1){$q_2$}
\pscircle*[linecolor=white](3.2,1){.18}
\pscircle(3.2,1){.18}
\uput{.3}[90](3.2,1){$q_3$}
\pscircle*[linecolor=white](5.4,1){.18}
\pscircle(5.4,1){.18}
\uput{.3}[90](5.4,1){$q_{n-1}$}
\pscircle*[linecolor=white](7,1){.18}
\pscircle(7,1){.18}
\uput{.3}[90](7,1){$q_{n}$}
\end{pspicture}
\end{center}}

The Garside element~$\Delta_{B_n}$ of~$\BBc(2,1,n) \sim \BBc(B_n)$ is the lcm of~$Q_n=\{q_1, q_2, {\ldots}, q_n\}$, which can be written in the various forms:
$$ \begin{array}{rcl}
\Delta_{B_n}
& = & q_1 (q_2 q_1 q_2) \cdots (q_{n} \cdots q_2 q_1 q_2 \cdots q_{n})\\
& = & (q_{n} \cdots q_2 q_1 q_2 \cdots q_{n}) \cdots (q_2 q_1 q_2) q_1\\
& = & (q_1 q_2 \cdots q_n)^n\\
& = & (q_n \cdots q_2 q_1)^n.
\end{array}$$
This is a central element  of~$\BBc(B_n)$.

Define $\hom:Q_{r-1}^* \rightarrow (\TE \cup\SS)^*$ by 
$\hom(q_i) = \left\{
\begin{array}{ll}
\tau & \mbox{for~$i = 1$, and}\\
s_{i+1} & \mbox{for~$i>1$.}
\end{array} \right.$

As mentioned in Subsection~\ref{sss:folding}, $\hom$ induces an injection~$\BBc(B_{r-1})\hookrightarrow\BBo(e,e,r)$; in particular, the poset structures with respect to \cleq\ coincide on~$\BBc(B_{r-1})$ and~$\BBo(e,e,r)$. We deduce that the element~$\Lambda=\hom(\Delta_{B_{r-1}})$ in~$\BBo(e,e,r)$ has the following decompositions:
$$\begin{array}{rclll}
\Lambda
& = & \tau  (s_3 \tau s_3)  \cdots (s_{r} \cdots s_4 s_3 \tau s_3 s_4 \cdots s_{r}) & = & (\t s_3 \cdots s_r)^{r-1} \\
& = &  (s_{r} \cdots s_4 s_3 \tau s_3 s_4 \cdots s_{r}) \cdots  (s_3 \tau s_3)  \t & = & (s_r \cdots s_3 \t)^{r-1}
\end{array}
$$
and is precisely the least common multiple of~$\hom(Q_{r-1}) = \{\tau\} \cup\SS$. 
Since $\tau$ is the lcm of~$\TE$, we deduce:

\begin{lemma}
\label{lambdaLcm}The element~$\Lambda$ is the lcm of~$\TE \cup\SS$.
\end{lemma}

\noindent
Also, by centrality of~$\Delta_{B_{r-1}}$ in~$\BBc(B_{r-1})$, we have
$$\Lambda\t = \t \Lambda\mbox{ and } \Lambda s_p = s_p \Lambda\mbox{ for } 3 \leq p \leq r.$$

Let~$\Lambda_2 = \tau$, and~$\Lambda_p =  s_p \cdots s_3 \tau s_3 \cdots s_p$ for~$3 \leq p \leq r$.
We find:\begin{center}
\fbox{\parbox{95pt}{
\vspace{2mm}{
$~~\Lambda = \Lambda_2 \Lambda_3 \cdots \Lambda_r.~~$}
\vspace{2mm}}}
\end{center}

A \emph{balanced element} in a monoid is an element~$\beta$ 
such that~$x \cleq \beta$ holds precisely when~$\beta \cgeq x$ holds. 
 
The following result could be deduced from older results~(see for instance~\cite{dfx} or~\cite{dehornoy:thin}), but the proof is straightforward and we include it to make the current work self-contained.

\begin{proposition}
\label{balbyaut}
Suppose that~$M$ is a cancellative monoid and $\beta$ is an element in~$M$ such that 
for all~$x \in M$ there exists an element~${\auto}(x)$ satisfying~$\beta x = \auto(x)\beta$.  
If $\auto$ is surjective then $\beta$ is balanced.
\end{proposition}

\begin{proof}
Suppose there exist~$x, y \in M$ satisfying~$\auto(x) = \auto(y)$. Then $\beta x = \beta y$ holds,
hence $x = y$ by left cancellation, so~$\auto$ is injective.
Thus~$\auto$ is an automorphism of~$M$.

For~$x \cleq \beta$, denote by~$\beta_x$
the unique
element of~$M$ satisfying~$x \  \beta_x = \beta$ (uniqueness comes from left cancellation).
Similarly, for~$\beta \cgeq x$ then write~$_x\beta \  x = \beta$.

Suppose~$x \cleq \beta$.  Then we have~$\beta = \auto(\beta) = \auto(x) \auto(\beta_x)$,
hence~$\auto(x)\cleq \beta$. 
By the same argument but using~$\auto^{-1}$ instead, we deduce that
$\auto(x)\cleq \beta$ implies~$x \cleq \beta$. Thus 
$x \cleq \beta$ holds precisely when~$\auto(x)\cleq \beta$ holds.
A symmetric argument shows that~$\beta \cgeq x$ holds precisely when~$\beta \cgeq \auto(x)$ holds.

So finally, suppose~$x \cleq \beta$, which implies~$\auto(x)\cleq \beta$.
Then we have
\[\auto(x)\beta_{\auto(x)} x = \beta x = \auto(x) \beta.\]
Left cancellation then gives~$\beta_{\auto(x)} x = \beta$, hence~$\beta \cgeq x$.
A similar argument shows that~$\beta \cgeq x$ implies~$x \cleq \beta$. 
Hence $\beta \cgeq x$ holds precisely when~$x \cleq \beta$ holds.
\end{proof}

\begin{proposition}
\label{bali}
The element~$\Lambda$ is balanced.
\end{proposition}

\begin{proof}
Let~$i\in\Z/e$. From~$(s_3 \t s_3) t_{i-1} = s_3 t_{i} t_{i-1} s_3 t_{i-1} = s_3 t_{i} s_3 t_{i-1} s_3 = t_{i} s_3 t_{i} t_{i-1} s_3 = t_{i} (s_3 \t s_3)$, we deduce~$\Lambda_2 t_{i-2} = t_i \Lambda_2$ and~$\Lambda_p t_{i}^\sda = t_i \Lambda_p$ for~$3 \leq p \leq r$,
hence~$\Lambda t_i = t_{i+r}\Lambda$.

Defining~$\auto(s_p) = s_p$ and~$\auto(t_i) = t_{i+r}$ for~$i\in\Z/e$ gives rise to an automorphism~$\auto$
of the cancellative monoid~$\BBo(e,e,r)$ satisfying~$\Lambda x = \auto(x)\Lambda$.
The result then follows by Proposition~\ref{balbyaut}.
\end{proof}

\subsubsection{Proof of Theorem~\ref{th:garside}.}
\label{sss:conclusion}
We now have enough to complete the proof of Theorem~\ref{th:garside}, that $\BBo(e,e,r)$ is a Garside monoid:

\begin{proof}[Proof of Theorem~\ref{th:garside}]  On the one hand, by Lemma~\ref{complete}, $\BBo(e,e,r)$ admits a complemented and complete presentation. On the other hand, the element~$\Lambda_2 \Lambda_3 \cdots \Lambda_r$---which we will henceforth denote by~$\Delta$---is the Garside element of~$\BBo(e,e,r)$. Indeed, Proposition~\ref{bali} states that left and right divisors of~$\Delta$ coincide and Lemma~\ref{lambdaLcm} insures that~$\Delta$ is the lcm of the generators and, in particular, the set of divisors of~$\Delta$ generates~$\BBo(e,e,r)$. Now, invoking Theorem~\ref{garsideCriterium}, we obtain that~$\BBo(e,e,r)$ is a Garside monoid with Garside element~$\Delta = \Lambda_2 \Lambda_3 \cdots \Lambda_r$.\end{proof}

Hence $\BBo(e,e,r)$ embeds in the group~$\BB(e,e,r)$ defined by the same presentation. Furthermore, we have for free:

\medskip

\textbf{Proposition \ref{MonoidIsomorphism}.}
\emph{The submonoid of~$\BB(e,e,r)$ generated by~$\TE\cup\SS$ is isomorphic to the monoid~$\BBo(e,e,r)$, that is, it can be presented by~$\PPo(e,e,r)$ considered as a monoid presentation.
}

\subsection{ Structure of the lattice of simples in~$\BBo(e,e,r)$}
\label{ss:lattice}

Here we completely describe the structure of the lattice of simples in the Garside monoid~$\BBo(e,e,r)$. Though sometimes somewhat technical, our careful study leads to a clear statement (Theorem~\ref{Simples}) which fully explains the combinatorics of the Garside structure and which will allow the computation of several related numerical objects and then, in the next subsection, to appreciate how classical $\BBo(e,e,r)$ actually is.

Since the only relations from~$\PPo(e,e,r)$ which can be applied
to~$\Lambda_k$ correspond to applications of~$(R_5)$ to~$\tau$, there are only four types of {non-trivial} left divisors of~$\Lambda_k$:
\begin{itemize}
\item[(1)] $s_k \cdots s_3 \tau s_3 \cdots s_{j-1}s_{j}$,
\item[(2)] $s_k \cdots s_3 \tau$,
\item[(3)] $s_ k \cdots s_3 t_i$,  and 
\item[(4)] $s_k \cdots s_{j+1}s_j$,
\end{itemize}
with~ $3\leq j\leq k$ and $i$ in~$\Z/e$.

The following theorem together with the fact that
we know precisely the form of the divisors of~$\Lambda_k$ for each~$k$
allows us to have precise control over the simples.

\begin{theorem}\label{Simples} 
The simples in~$\BBo(e,e,r)$ are precisely the elements of the form~$p_2 \cdots p_r$ 
where~$p_k$ is a divisor of~$\Lambda_k$ for~$2\leq k\leq r$.
\end{theorem} 

Define the polynomial $P^\oplus_{(e,e,r)}(q) = \sum a_n q^n$ where $a_n$ is the number of  length~$n$ simples in~$\BBo(e,e,r)$.
This polynomial is discussed in more detail in Subsection~\ref{sss:Poincare}.
Theorem~\ref{Simples} directly gives a factorization of it.

\begin{corollary} 
\label{corpp} We have:
$$\displaystyle P^\oplus_{(e,e,r)}(q)
=\prod_{k=2}^{r}(1+q+\cdots+q^{k-2}+eq^{k-1}+q^{k}+\cdots+q^{2k-2}).$$
\end{corollary}

\begin{corollary} The number of simples in~$\BBo(e,e,r)$ is
\[P^\oplus_{(e,e,r)}(1)=\prod_{k=2}^r(2(k-1)+e)=\frac{(2(r-1)+e)!!}{e!!},\]
where the notation~$n!!$ represents the product~$n(n-2)\cdots4\cdot2$ for~$n$ even and the product~$n(n-2)\cdots5\cdot3$ for~$n$ odd (see~\cite[sequences~A000165 and~A001147]{slo}).
\end{corollary}

{\bf Remark.} For~$q\neq 1$, we find
\begin{eqnarray*}
P^\oplus_{(e,e,r)}(q)
&=&\prod_{k=2}^{r}\frac{q^{2k-1}+(e-1)q^k-(e-1)q^{k-1}-1}{q-1}.
\end{eqnarray*}

For instance, Figure~\ref{C:latticeB333} displays the lattice of simples in~$\BBo(3,3,3)$. We find:
\begin{eqnarray*}
P^\oplus_{(3,3,3)}(q)&=&1+4q+7q^2+11q^3+7q^4+4q^5+q^6\\
&=&(1+3q+q^2)(1+q+3q^2+q^3+q^4).
\end{eqnarray*}

\begin{figure}[htb]
	\begin{center}
		\includegraphics[height=190pt]{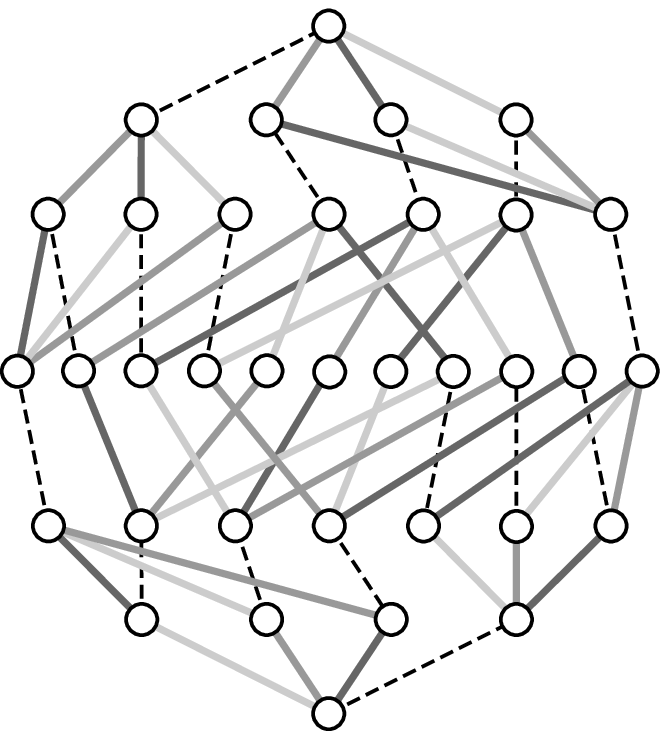}
		\put(-86.7,-6){\makebox(0,0){$1$}}
		\put(-60,26){\makebox(0,0){$t_2$}}
		\put(-92,25){\makebox(0,0){$t_1$}}
		\put(-139,20){\makebox(0,0){$t_0$}}
		\put(-170,50){\makebox(0,0){$\Lambda_2$}}%
		\put(-1,140){\makebox(0,0){$\Lambda_3$}}%
		\put(-30,20){\makebox(0,0){$s_3^{}$}}
		\put(-86.7,196){\makebox(0,0){$\Delta$}}
	\end{center}
\caption{The lattice of simples in~$\BBo(3,3,3)$.}
\label{C:latticeB333}
\end{figure}

To prove Theorem~\ref{Simples}, we will need the two lemmas below, which make use of the notation of the \emph{height} of an element of~$\BBo(e,e,r)$ :

Define a map~$\Ht: (\TE\cup\SS)^* \rightarrow\{1,2, \ldots, r\}$ on generators by~$\Ht(t_i) = 2$, $\Ht(s_q) = q$, and, for~$w = a_1 \cdots a_k$,  define $\Ht(w) = \max_{i=1}^k \Ht(a_i)$.  Define the height of the empty word to be~$\Ht(\varepsilon) = 1$. If $\rho_1 = \rho_2$ is a defining relation of~$\BBo(e,e,r)$, it is clear from inspection that $\Ht(\rho_1) =  \Ht(\rho_2)$ holds. Thus $\Ht(w)$ only depends on the element in~$\BBo(e,e,r)$ represented by~$w$. So the \emph{height} map~$\Ht: \BBo(e,e,r)\rightarrow\{1,2, \ldots, r\}$ is well-defined.

\begin{lemma}\label{KeyLemma} Every~$x \in \BBo(e,e,r)$ with~$\Ht(x) < k$ satisfies
$$x \Lambda_k=\Lambda_k x^\sda.$$
\end{lemma}

\begin{proof} For~$j<k$, we find:
\begin{eqnarray*}
s_j\Lambda_k
&=&{ s_j}s_k\cdots s_{j+1}s_js_{j-1}\cdots s_3\tau s_3\cdots s_k\\
&=&s_k\cdots s_{j+2}{ s_j}s_{j+1}s_js_{j-1}\cdots s_3\tau s_3\cdots s_k\\
&=&s_k\cdots s_{j+2}{ s_{j+1}s_js_{j+1}}s_{j-1}\cdots s_3\tau s_3\cdots s_k\\
&=&s_k\cdots s_3{ s_{j+1}}\tau s_3\cdots s_k\\
&=&s_k\cdots s_3\tau{ s_{j+1}}s_3\cdots s_{j-1}s_js_{j+1}\cdots s_k\\
&=&s_k\cdots s_3\tau s_3\cdots s_{j-1}{ s_{j+1}}s_js_{j+1}\cdots s_k\\
&=&s_k\cdots s_3\tau s_3\cdots s_{j-1}{ s_js_{j+1}s_j}s_{j+2}\cdots s_k\\
&=&s_k\cdots s_3\tau s_3\cdots s_{j-1}{ s_js_{j+1}s_j}s_{j+2}\cdots s_k\\
&=&s_k\cdots s_3\tau s_3\cdots s_{j-1}s_js_{j+1}s_{j+2}\cdots s_k{ s_j}\\
&=&\Lambda_ks_j.
\end{eqnarray*}
For~every~$k$ and~$0\leq i<e$, we have:
\begin{eqnarray*}
t_i\Lambda_k
&=&{ t_i}s_k\cdots s_3\tau s_3\cdots s_k\\
&=&s_k\cdots s_4{ t_i}s_3{ t_it_{i-1}}s_3\cdots s_k\\
&=&s_k\cdots s_4{ s_3t_is_3}t_{i-1}s_3\cdots s_k\\
&=&s_k\cdots s_3t_i{ t_{i-1}s_3t_{i-1}}s_4\cdots s_k\\
&=&s_k\cdots s_3t_it_{i-1}s_3s_4\cdots s_k{ t_{i-1}}\\
&=&\Lambda_kt_{i-1}.
\end{eqnarray*}
The result follows.\end{proof}

Recall that~$a\vee b$ denotes the lcm of~$a$ and~$b$ and that, by cancellativity, elements~$(a\backslash b)$ and~$(b\backslash a)$ are uniquely defined by:
$$a \vee b =a(a\backslash b)=b(b\backslash a).$$

\begin{lemma}\label{TechLemma} Let~$a$ be an element in~$\TE\cup\SS$ and~$q_k$ be a right divisor of some~$\Lambda_k$. Then~$a\wedge q_k=1$ and~$\Ht(q_k\backslash a)<k$ together imply~$q_k\backslash a\in\TE\cup\SS$ and~$a\backslash q_k=q_k$.\end{lemma}

\begin{proof} 
If $q_k$ is trivial, then the result follows directly. Consider the remaining cases:
\begin{enumerate}
\item[(1)] Let~$q_{k} = s_j \cdots s_{k}$ for some~$3 \leq j \leq k$. 
\newline
First, $a \not \preccurlyeq q_{k}$ implies~$a\neq s_j$. Next, $s_{j-1}\vee q_{k}=q_{k} s_{j-1} q_{k}$
(\resp $t_i\vee q_{k} =  q_{k} t_i q_{k}$ for~$j=3$) and~$\Ht(q_{k}) =k$
imply~$a\neq s_{j-1}$ (\resp $a\neq t_i$ for~$j=3$). The only possible cases are then:
$$a\vee q_{ k}= \left\{ 
\begin{array}{llll}
s_l q_{k} & = & q_{k} s_l  & \mbox{ for } a = s_l \mbox{ with } 3 \leq l < j-1, \\
s_l q_{k} & = & q_{k} s_{l-1}  & \mbox{ for } a = s_l \mbox{ with } j < l \leq k,\mbox{ and }\\
t_i q_{k} & = & q_{k} t_i  & \mbox{ for } a = t_i \mbox{ and } j \neq 3.
\end{array}
\right.$$
\item[(2)] Let~$q_{k} = t_i s_3 \cdots s_{k}$ for some~$i \in \Z/e$.
\newline
First, $a \not \preccurlyeq q_{k}$ implies~$a\neq t_i$. Next, 
$t_i\vee q_{k} = q_{k} t_{i-1} q_{k}$ and~$\Ht(q_{k}) =k$ imply~$a\neq t_j$ for~$j \neq i$. The possible cases are then:
$$a\vee q_{k} = \left\{ 
\begin{array}{llll}
s_3 q_{k} & = & q_{k} t_i & \mbox{ for } a = s_3, \mbox{ and }\\
s_l q_{k} & = & q_{k} s_{l-1}  & \mbox{ for } a = s_l \mbox{ with } 3 < l \leq k.\phantom{\mbox{ and }}
\end{array}
\right.$$
\item[(3)] Let~$q_{k} = s_{j-1} \cdots s_3 \tau s_3 \cdots s_{k}$ for some~$3 \leq j \leq k$.
\newline
First, $a \not \preccurlyeq q_{k}$ implies~$a\neq s_{j-1}$ (\resp $a\neq t_i$ for~$j=3$).
Next, $s_j\vee q_{k} = q_{k} s_j q_{k}$ and~$\Ht(q_{k}) = k$ imply~$a\neq s_j$.
The possible cases are then:
$$a\vee q_{k} = \left\{ 
\begin{array}{llll}
s_l q_{k} & = & q_{k} s_l& \mbox{ for } a = s_l \mbox{ with } 3 \leq  l < j-1, \\
s_l q_{k} & = & q_{k} s_{l-1}  & \mbox{ for } a = s_l \mbox{ with } j < l \leq k, \mbox{ and } \\
t_i q_{k} & = & q_{k} t_i & \mbox{ for } a = t_i \mbox{ and } j \neq 3.
\end{array}
\right.$$
\end{enumerate}
In each case, we find~$a\vee q_{k} = a q_{k} = q_{k} a'$
with~$a' = q_k\backslash a  \in\TE \cup\SS$ and~$a \backslash q_k = q_k$. 
\end{proof}

\begin{proof}[Proof of Theorem \ref{Simples}] We have to prove a {double inclusion}.
First, we show that if, for each~$k\in\{2,\ldots,r\}$, $p_k$ is a divisor of~$\Lambda_k$, then $p_2 p_3 \ldots p_r$ divides~$\Di$. For each~$k$, let~$q_k$ be the unique element of the monoid satisfying~$p_kq_k= \Lambda_k$.
Lemma~\ref{KeyLemma} implies $\Lambda_j q_k^\sda = q_k \Lambda_{j}$ for each~$q_k$ and each~$j>k$.
Let~$q'_k$ be the element obtained by applying $r-k$ times the map~$\sda$ to~$q_k$. Then we obtain
$$\Di = (p_2 q_2) \Lambda_3 \cdots \Lambda_r = p_2  \Lambda_3 \cdots \Lambda_r q'_2
= \cdots = p_2 p_3 \cdots p_r q_r q'_{r-1} \cdots q'_3 q'_2.$$
Thus~$p_2 p_3 \cdots p_r$ is a divisor of~$\Di$. This completes the first inclusion.

\bigskip

Now let~$k\in\{2,\ldots,r\}$. We prove, by induction on~$k \geq 2$,  that if~$p$ left-divides~$\Lambda_2 \cdots \Lambda_k$ then $p = p_2 \cdots p_k$ holds for some divisors~$p_j$ of~$\Lambda_j$ with~$2\leq j\leq k$. The result holds vacuously for~$k=2$.
 
Assume~$k>2$. By the induction hypothesis, $p\wedge\Lambda_2 \cdots \Lambda_{k-1}$ can be written as $p_2 \cdots p_{k-1}$, where $\Lambda_j$ is~$p_j q_j$ for some~$q_j$ with~$\Ht(q_j)\leq j$. 
We obtain~$p = p_2 \cdots p_{k-1} p_0$ for some 
$p_0 \cleq q'_{k-1} \cdots q'_2 \Lambda_k= \Lambda_kq'^\sda_{k-1} \cdots q'^\sda_2$ where~$q'_j$ is, as above, the element obtained by applying the map~$\sda$, $(k-j)$ times to~$q_j$.
Let~$p_k=p_0\wedge\Lambda_k$, $p_kq_k= \Lambda_k$ and~$p_0 = p_kp'$. 

\begin{figure}[h]
	\begin{center}
		\hspace*{-45pt}
		\includegraphics[width=341.5pt]{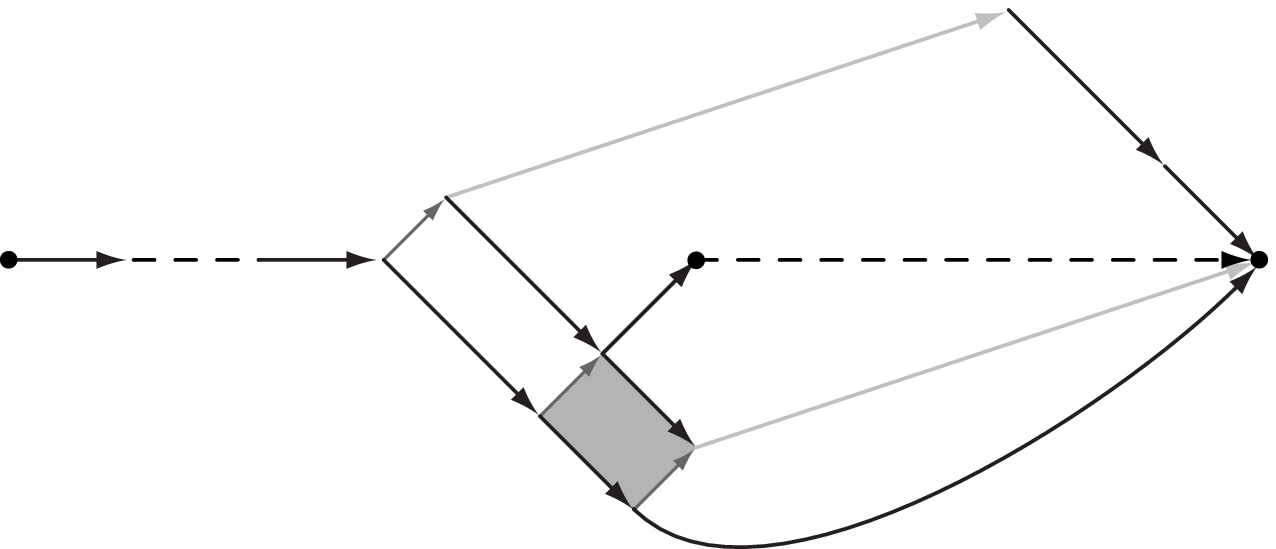}
		\put(-339,68.5){\makebox(0,0){\black$1$}}
		\put(-326,87.5){\makebox(0,0){$p_2$}}
		\put(-258,87.5){\makebox(0,0){$p_{k-1}$}}
		\put(-225,50){\makebox(0,0){$p_k$}}
		\put(-192,17,5){\makebox(0,0){$q_k$}}
		\put(-195,80){\makebox(0,0){$p_k$}}
		\put(-162,47,5){\makebox(0,0){$q_k$}}
		\put(-173,71){\makebox(0,0){$p'_a$}}
		\put(-193,47,5){\makebox(0,0){$a$}}
		\put(-161,14,5){\makebox(0,0){$b$}}
		\put(-235,89,5){\makebox(0,0){$b^\sua$}}
		\put(-177,34){\makebox(0,0){\black \tiny Lemma}}
		\put(-174,26){\makebox(0,0){\black \tiny \ref{TechLemma}}}
		\put(-154,85.5){\makebox(0,0){\black$p$}}
		\put(-46,132){\makebox(0,0){$p_k$}}
		\put(-13,96,5){\makebox(0,0){$q_k$}}
		\put(-75,14,5){\makebox(0,0){$q$}}
		\put(17,68){\makebox(0,0){\black$\Lambda_2\cdots\Lambda_k$}}
		\end{center}
\caption{Proof of Theorem~\ref{Simples}}
\label{ProofSimples}
\end{figure}

We have to show that $p'$ is trivial.
Suppose instead that $p'$ is not trivial. Then we may write
$p' = ap'_a$ for some~$a\in\TE\cup\SS$. 
We have $p' \cleq  q_kq$ with~$q = q'^\sda_{k-1} \cdots q'^\sda_2$ and~$\Ht(q)<k$. 
Thus $a\vee q_k$ must divide~$q_kq$, but 
$a \not \preccurlyeq  q_k$ holds by $\gcd$-ness of~$p_k$. Therefore, by Lemma~\ref{TechLemma}, there exists~$b\in\TE\cup\SS$ satisfying~$aq_k=q_kb$, thus~$p_kaq_k=\Lambda_kb$, hence~$p_ka=b^\sua p_k$. We find~$p=p_2\cdots p_{k-1}b^\sua p_kp'_a$ with~$p_2\cdots p_{k-1}b^\sua\cleq\Lambda_2\cdots\Lambda_{k-1}$, which contradicts~$p_2\cdots p_{k-1}=p\wedge\Lambda_2\cdots\Lambda_{k-1}$. Therefore, $p'$ is trivial, which concludes the induction.
\end{proof}

\subsubsection*{ A note on the reflection group $\GG(e,e,r)$ and the Garside structure}

The number of simples in $\BBo(e,e,r)$ is~$\displaystyle \frac{(2(r-1) + e)!!}{e!!}$, while the number of elements in $\GG(e,e,r)$ is~$e^{r-1} r!$. For~$e=2$, we have equality between these two expressions, corresponding to the classical type~$D_r$ case. For~$e>2$, we have~$\frac{x+e}{e} < \frac{x}{2} + 1$, hence
$$ 
\begin{array}{rcl}
\displaystyle\frac{(2(r-1) + e)!!}{e!!} &= & \displaystyle e^{r-1} \left(\frac{2(r-1)+e}{e}\right) \cdots \left(\frac{4+e}{e}\right)
\left(\frac{2+e}{e}\right) \\
&< &e^{r-1} r (r-1) \cdots (3)(2) = e^{r-1}r!
\end{array}$$ 
(For example, there are 35~simples in $\BBo(3,3,3)$---see Figure~\ref{C:latticeB333}---and $54$~elements in $\GG(3,3,3)$.) In other words, not all elements of $\GG(e,e,r)$ may be represented by simples from $\BBo(e,e,r)$. For example, the element
$$\overline{t_0t_1} = 
\left(\begin{array}{c|c}
 \begin{array}{rl}
 \zeta_e & 0 \\
0 &  \zeta_e^{-1} \\
\end{array}
& \begin{array}{c} \\ {\LARGE 0} \\ \\ \end{array} \\
\hline 
\\
0  & ~ {\LARGE I_{r-2}} ~ \\
\\
\end{array} \right),$$
may not be represented by a simple from $\BBo(e,e,r)$.

The known classical braid monoids for braid groups of real reflection groups all have equality between number of simples and size of reflection group. In this way the monoid~$\BBo(e,e,r)$ appears not to be strictly classical. However in a number of ways it is seen to be dual, or simply different from, the so-called dual braid monoids, and so deserves a name like post-classical. This is the content of the next subsection.

\subsection{ How classical is $\BBo(e,e,r)$?}
\label{ss:howClassical}

The braid group~$\BB(e,e,r)$ seems to admit no classical braid monoid, in the sense that its submonoid generated by the generators of~\cite{bmr}---providing a minimal generating set---is indeed not finitely presented (see~\cite{co,bcArXiv,bc}). Recall that the braid monoid~$\BBo(e,e,r)$ (\resp the dual braid monoid~$\BBd(e,e,r)$) coincides with the classical braid monoid~$\BBc(D_r)$ (\resp the dual braid monoid~$\BBd(D_r)$) for~$e=2$ and with the dual braid monoid~$\BBd(I_2(e))$ for~$r=2$. 

In this subsection we look at various properties of the monoid~$\BBo(e,e,r)$, which mainly deal with enumerative aspects, and consider them in relation to known classical and dual braid monoids for other braid groups. While it cannot be considered as strictly classical, $\BBo(e,e,r)$ has much in common with the classical braid monoids than with the dual braid monoids, and it could be considered as a \emph{post-classical} braid monoid. The following three observations allow to legitimate this terminology. 

\subsubsection{ A kind of duality}
\label{sss:duality}

According to~\cite{bessisdual}, the duality terminology in the context of Garside monoids for braid groups of finite real reflection groups~$W$ can be justified by the numerical facts summarized in the following table:\footnote{Each of the braid group presentations constructed in~\cite{bessiszariski} corresponds to a regular degree~$d$. The product of the generators raised to the power~$d$ (which is the order of the image of this product in the reflection group), is always central. See also~\cite{bessistopology}.}

\[
\setstretch{1.2}
\begin{array}{ccc}
\toprule
								& \quad\BBc(W)\quad	& \quad\BBd(W)\quad \\
\midrule
\mbox{Product of the atoms}			& \bc					& \bw_0 \\
\Delta 							& \bw_0				& \bc \\
\midrule
\mbox{Number of atoms}				& n 					& N \\
\mbox{Length of~$\Delta$}			& N					& n \\
\midrule
\mbox{Order of~$a\mapsto a^\Delta$}	& 2					& h \\
\mbox{Regular degree}				& h					& 2 \\
\bottomrule
\end{array}
\]

\def\duality{\hbox{\boldmath$(r-1)$}}

\bigbreak
\vbox{A different kind of duality can be observed between the monoids~$\BBo(e,e,r)$ and~$\BBd(e,e,r)$:
\[
\setstretch{1.4}
\begin{array}{ccc}
\toprule
						& \BBo(e,e,r)	& \BBd(e,e,r)\\
\midrule
\mbox{Number of atoms}		& e+r-2		& (e+r-2)\duality \\
\mbox{Length of $\Delta$}	& r\duality		& r \\
\mbox{Order of $\displaystyle a\mapsto a^\Delta$}
						& \displaystyle\frac{e}{e\wedge r}
									& \displaystyle\frac{e\duality}{e\wedge r}\\
\bottomrule
\end{array}
\]}

Thus the monoid~$\BBo(e,e,r)$ may be considered to be a kind of dual of the dual braid monoid~$\BBd(e,e,r)$. The latter fits into the general framework of dual braid monoids defined in~\cite{bessisdual}, but it satisfies only some of the numerical properties summarized in the first table above. In a parallel way, $\BBo(e,e,r)$ could be named simply classical. Here, we could mention that neither~$\BBo(e,e,r)$ nor~$\BBd(e,e,r)$ can be produced by~\cite[Theorem~0.1]{bessiszariski}, so in particular, the notion of regular degree is not relevant.

\subsubsection{ Poincar\'e polynomial}
\label{sss:Poincare}

For a given Garside monoid~$M$, the polynomial $P_M
$ is defined by $P_M(q)  = \sum a_n q^n$  where $a_n$ denotes the number of length~$n$ simples in~$M$ (see earlier comments preceding Corollary~\ref{corpp}). 
In the case of the classical braid monoids associated to finite Coxeter group~$W$ (for example $\BBc(A_n)$, $\BBc(B_n)$, $\BBc(D_n)$, etc), this polynomial coincides with the Poincar\'e polynomial of~$W$, where $a_n$ is the number of length~$n$ elements of with respect to a set of simple reflections. In these cases,
we have:
\[P^+_{W}(q)=\prod_{k=1}^{r}(1+q+\cdots+q^{d_k-1})\]
where the numbers $d_k$ denote the reflection degrees.
The polynomial~$P^\oplus_{(e,e,r)}(q)$ does not satisfy this general formula, except for the cases~$e=2$ or~$r=2$. However the similarity of factorization of the Poincar\'e polynomial (see below) suggests describing~$\BBo(e,e,r)$ again as a post-classical braid monoid.
 
\begin{eqnarray*}
P^\oplus_{(e,e,r)}(q)&=&\prod_{k=1}^{r}(1+q+\cdots+q^{k-2}+eq^{k-1}+q^{k}+\cdots+q^{2k-2}),\\
P^+_{A_n}(q)&=&\prod_{k=1}^{n}(1+q+\cdots+q^k),\\
P^+_{B_n}(q)&=&\prod_{k=1}^{n}(1+q+\cdots+q^{2k-1}),\\
P^+_{D_n}(q)&=&(1+q+\cdots+q^{n-1})\prod_{k=1}^{n-1}(1+q+\cdots+q^{2k-1}).
\end{eqnarray*}

\subsubsection{ Zeta polynomial}
\label{sss:Zeta}

For a given Garside monoid~$M$, the \emph{zeta polynomial $Z_M$} can be defined by requiring that $Z_M(q)$ be the number of length~$q-1$~multichains~$a_1\cleq\cdots\cleq a_{q-1}$ in the lattice of simples of~$M$. Whenever $\GG(de,e,r)$ is well-generated (which is the case for $\GG(e,e,r)$), the zeta polynomial of the dual braid monoid~$\BBd(de,e,r)$ admits a nice factorization:
\[Z^\times_{(de,e,r)}(q)=\prod_{k=1}^{r}\frac{d_k+d_r(q-1)}{d_k},\]
where $d_1\leq\cdots\leq d_r$ are the reflection degrees (see~\cite{chap,reiner,athareiner}). On the contrary, the zeta polynomial~$Z^+_{(de,e,r)}$ of the classical braid monoid~$\BBc(de,e,r)$ (when defined) is not known to admit any nice factorization. In this way, $\BBo(e,e,r)$ has more in common with classical braid monoids than dual braid monoids. For instance, we find:
\begin{eqnarray*}Z_{(3,3,3)}^\oplus(q)
&=&\frac{11q^6+171q^5+985q^4+2585q^3+2964q^2+1444q+240}{240}\\
&=&\frac{(q+1)(q+6)(11q^4+94q^3+261q^2+194q+40)}{240}.\end{eqnarray*}

\subsection{ Conclusion}
\label{ss:conclusion}

While we feel that the new Garside monoid~$\BB(e,e,r)$ deserves the description post-classical, we do not exclude the possibility that this is the best presentation available, and would like to conclude with a motivating question:

\begin{question} Does $\BB(e,e,r)$ admit other Garside structures?
\end{question}

\section{ Acknowledgment}

The completion of this work was made possible from a  collaboration begun at the GDR Tresses conference in Autrans, 2004 (GDR 2105 CNRS, ``Tresses et Topologie de basse dimension" \cite{gdr}).

The first author would also like to thank the European Union for a Marie Curie Postdoctoral Research Award at the time this work was undertaken.

The authors thank Ivan Marin for pointing out to them that, during his thesis work supervised by~Daan Krammer, Mark Cummings had discovered the same Garside structure for~$\BB(e, e, r)$, even though proofs and motivations are essentially different.

\end{document}